\documentclass[a4paper,12pt]{article}
\usepackage{setspace}
\usepackage[top=2.5cm,bottom=2.5cm,left=2.5cm,right=2.5cm]{geometry}
\usepackage{cite, amsmath, amssymb,amsfonts,mathrsfs,amsthm}
\usepackage{latexsym}
\usepackage{hyperref}
\usepackage{graphicx,booktabs,multirow}
\usepackage{tikz}
\usetikzlibrary{decorations.pathreplacing}
\usetikzlibrary{intersections}
\usepackage{enumerate}
\usepackage{appendix}
\usepackage{graphicx}
\usepackage{subfigure}
\newtheorem{theorem}{Theorem}[section]
\newtheorem{proposition}[theorem]{Proposition}
\newtheorem{lemma}[theorem]{Lemma}
\newtheorem{corollary}[theorem]{Corollary}
\newtheorem{remark}[theorem]{Remark}
\newtheorem{example}[theorem]{Example}

\newtheorem{question}[theorem]{Question}
\newtheorem{definition}[theorem]{Definition}
\numberwithin{equation}{section}
\allowdisplaybreaks[4]

\usepackage[section]{placeins}
\def\NAT@def@citea{\def\@citea{\NAT@separator}}
\allowdisplaybreaks

\begin{document}
	
	\noindent
	\begin{center}
			{\Large \bf On the characterizations of $\mathcal{D}_{k} \backslash \mathcal{D}_{k-2}$}
	\end{center}
	\vspace*{7mm}
	
	\noindent
	{\large \bf Jing Zeng$^{\,a}$, Lihua You$^{\,b,*}$, Xinghui Zhao$^{\,b}$, Hong-Jian Lai$^{\,c,d,*}$}
	\noindent
	
	\vspace{5mm}
	
	\noindent
	$^{a}$ College of Cryptology and Cyber Science, Nankai University, Tianjin, 300350, P. R. China.
	
	\noindent
	$^{b}$ School of Mathematical Sciences, South China Normal University,  Guangzhou, 510631, P. R. China.
	
	\noindent
	$^{c}$ School of Mathematics and Systems Science, Guangdong Polytechnic Normal University, Guangzhou 510665, P. R. China.
	
	\noindent
	$^{d}$ Department of Mathematics, West Virginia University, Morgantown, WV, USA.\\
	E-mail: {\tt jing.zeng@yeah.net(Jing Zeng)}, \,{\tt ylhua@scnu.edu.cn(Lihua You)}, 
	
	\quad\quad{\tt 2025010152@m.scnu.edu.cn(Xinghui Zhao)}, 
	
	 \quad\quad{\tt hjlai2015@hotmail.com(Hong-jian Lai)} \\[2mm]
	$^*$ Corresponding author
	\noindent
	
	\vspace{7mm}
	
	\noindent
	{\bf Abstract} \
	\noindent
	The determinant of a tournament $T$, denoted by $\det(T)$, is defined as the determinant of the skew-adjacency matrix of $T$. It is well-known that $\det(T)$ is equal to $0$ if $n$ is odd, and $\det(T)$ is the square of an odd integer if $n$ is even. For a positive odd integer $k$, let $\mathcal{D}_k$ be the set of tournaments whose all subtournaments have determinant at most $k^2$. Former studies showed that for $k \in \{1,3,5\}$, a tournament $T \in \mathcal{D}_k \backslash \mathcal{D}_{k-2}$ ($T \in \mathcal{D}_1$ when $k=1$) if and only if $T$ is switching equivalent to a transitive blowup of $L_{k+1}$, where $L_{k+1}$ is a tournament of order $k+1$ with a specific structure. For $k \geq 7$, no characterization results are known. It was shown in [Discrete Math. 349 (2) (2026) 114766] that there exists a tournament $T \in \mathcal{D}_{7}\backslash \mathcal{D}_{5}$ which can not be switching equivalent to a transitive blowup of $L_{8}$. A natural problem (raised in the aforementioned paper as an open problem) is to characterize tournaments in $\mathcal{D}_{k} \backslash \mathcal{D}_{k-2}$ that can be switching equivalent to a transitive blowup of $L_{k+1}$ for $k \geq 7$.
	
	To address this problem and to further explore the structural properties of tournaments in $\mathcal{D}_{k}$, we introduce CR tournaments, strong CR tournaments, basic tournaments and $Z$-matrices, and investigate their properties. We use these properties to characterize those tournaments $T \in \mathcal{D}_{k} \backslash \mathcal{D}_{k-2}$ where $T$ contains a subtournament switching isomorphic to a basic strong CR tournament in $\mathcal{D}_{k} \backslash \mathcal{D}_{k-2}$. This result implies former characterizations of $\mathcal{D}_3\backslash \mathcal{D}_1$ and $\mathcal{D}_5 \backslash \mathcal{D}_3$.
	Using $Z$-matrices, we also show that for even $n$, $L_{n}$ is a basic strong CR tournament, and thus solve the open problem posed in  [Discrete Math. 349 (2) (2026) 114766].
	\\[2mm]
	\noindent
	{\bf Keywords:} \ Tournament; CR tournament; Skew-adjacency matrix; Determinant; Transitive blowup
	
	\noindent
	{\bf MSC:} \ 05C20, 05C50, 05C75

	\section{Introduction}
	
	\hspace{1.5em}A \textit{tournament} is a directed graph with exactly one arc between each pair of vertices. We denote a tournament of order $n$ by $n$-tournament. Let $T$ be an $n$-tournament with vertex set $\{v_1,\ldots,v_n\}$. If the arc between $v_i$ and $v_j$ is directed from $v_i$ to $v_j$ (resp. from $v_j$ to $v_i$), we say $v_i$ dominates $v_j$ (resp. $v_i$ is dominated by $v_j$), and write $v_i\rightarrow v_j$ (resp. $v_i \leftarrow v_j$). In this paper, we use $M^{\mathsf{T}}$ to denote the transpose of a matrix $M$. The \textit{adjacency matrix} of an $n$-tournament $T$, with respect to the vertex ordering $v_1,v_2,\ldots,v_n$,  is the $n\times n$ matrix $A_{T}=[a_{ij}]$ in which $a_{ij}=1$ if $v_i\rightarrow v_j$ in $T$ and $a_{ij}=0$ otherwise, and the \textit{skew-adjacency matrix} of an $n$-tournament $T$, is the $n\times n$ matrix $S_{T}=A_{T}-A_{T}^{\mathsf{T}}$. By the definition, $S_T$ is a skew-symmetric matrix, say, $S_T+S^{\mathsf{T}}_T=\mathbf{0}$. The \textit{determinant} of a tournament $T$, denoted by $\det(T)$, is defined as the determinant of $S_T$. It is easy to see that the determinant of $S_T$ remains constant under different vertex orderings. A well-known result of Cayley \cite{Pfaffian} showed that the determinant of a skew-symmetric matrix of even order is the square of its Pfaffian. Based on the properties of Pfaffian, for an $n$-tournament, Fisher and Ryan \cite{DET} showed that $\det(T)=0$ if $n$ is odd and $\det(T)$ is the square of an odd integer if $n$ is even.
	
	Throughout this paper, we use $V(T)$ to denote the vertex set of tournament $T$, and $|V(T)|$ to denote the number of vertices of $T$.  For $ X \subseteq V(T) $, we denote by $ T[X] $ the subtournament of $ T $ induced by $ X $.
	
	A tournament is a transitive tournament if it contains no directed cycles, or equivalently, if it is possible to order its vertices as $v_1,\ldots,v_n$ such that $v_i\rightarrow v_j$ if and only if $i<j$. Moreover, an equivalent assertion to $T$ contains no directed cycles is that $T$ contains no $3$-cycles. 
	
	A \textit{switch} of a tournament $ T $, with respect to a subset $ W $ of $ V=V(T) $, is the tournament obtained by reversing all the arcs between $ W $ and $V\backslash W$ (If $W=\emptyset$ or $W=V$, then the switch of $T$ is $T$ itself). If $T'$ is a switch of $T$, we say $T'$ and $T$ are switching equivalent. Two tournaments $T_1$ and $T_2$ with the same vertex set are switching equivalent if and only if their skew-adjacency matrices are $\{\pm1\}$-diagonally similar\cite{TWOGRAPH}. Hence, the determinant of a tournament $T$ is an invariant under switching operation. Moreover, if $T_1$ is switching equivalent to $T_2$ and $T_2$ is switching equivalent to $T_3$, then $T_1$ is switching equivalent to $T_3$.
	
	\begin{definition}\label{defswitchiso}
		A tournament $T_1$ is switching isomorphic to $T_2$ if there exists a switch of $T_1$, denoted by $T_1'$, such that $T_1'$ is isomorphic to $T_2$.
	\end{definition}
	
	Clearly, if $T_1$ is switching isomorphic to $T_2$, then $T_2$ is switching isomorphic to $T_1$. In particular, if $T_1$ and $T_2$ are switching equivalent, then $T_1$ is switching isomorphic to $T_2$.
	
	Let $X$ and $Y$ be two non-empty vertex sets. If $u\rightarrow v$ for any $u\in X$ and any $v\in Y$, we write $X\rightarrow Y$.

\begin{definition}{\rm(\!\!\cite{Dfive})}\label{DefBLWOUP}
	Let $T$ be an $n$-tournament with vertices $v_1,\ldots,v_n$, $H_1,\ldots,H_n$ be tournaments. A tournament $T(H_1,\ldots,H_n)$ is obtained by replacing each vertex $v_i$ with the tournament $H_i$ for each $1\leq i\leq n$, and adding arcs between $V(H_i)$ and $V(H_j)$ such that $V(H_i)\rightarrow V(H_j)$ if $v_i\rightarrow v_j$ for $1\leq i,j \leq n$, we call such $T(H_1,\ldots,H_n)$ is a blowup of $T$ with respect to $H_1,\ldots,H_n$.	
\end{definition}

	Follow the notation in Definition \ref{DefBLWOUP}, if $H_i$ is transitive for each $1\leq i\leq n$ and $a_i$ = $|V(H_i)|$, we call $T(H_1,\ldots,H_n)$ the $transitive$ $(a_1,\ldots,a_n)$-$blowup$ of $T$ (transitive blowup of $T$ for short), denoted by $T(a_1,\ldots,a_n)$\cite{BLOWUP}. Moreover, if $a_i=2$ for some $1\leq i\leq n$ and $a_j=1$ for each $j\in \{1,2,\ldots,n\}\backslash \{i\}$, we say $T(a_1,a_2,\ldots,a_n)$ is a \textit{1-transitive blowup} of $T$, where $a_1+a_2+\cdots +a_n=n+1$.
	
	For a positive odd integer $k$, let $\mathcal{D}_k$ be the set consisting of tournaments whose all subtournaments have determinant at most $k^2$ \cite{DTHREE} (note that $\det(T)$ is equal to $0$ or the square of an odd integer). Equivalently, a tournament $T\in \mathcal{D}_k$ if and only if all the principal minors of $S_T$ do not exceed $k^2$. Clearly, $\mathcal{D}_{k}$ is closed under switching operations.
	
	Let $k\geq 3$. The notation $T\in \mathcal{D}_k\backslash \mathcal{D}_{k-2}$ implies $T\in \mathcal{D}_k$ and $T\notin \mathcal{D}_{k-2}$. It is easy to see that $\mathcal{D}_{k} = (\mathcal{D}_{k} \backslash \mathcal{D}_{k-2}) \cup (\mathcal{D}_{k-2} \backslash \mathcal{D}_{k-4}) \cup \cdots \cup (\mathcal{D}_{3} \backslash \mathcal{D}_{1}) \cup \mathcal{D}_1$ for odd $k$. For convenience, we use $\mathcal{D}_1\backslash\mathcal{D}_{-1}$ to denote the set $\mathcal{D}_1$ in this paper.
	
	A \textit{diamond} is a $4$-tournament consisting of a vertex dominating or dominated by a 3-cycle, and a $4$-tournament is a diamond if its determinant is $9$ \cite{DIACHA}.
	
	The tournament $L_n$ $(n \geq 2)$ is an $n$-tournament in which there exists an ordering of vertices, $v_1,v_2,\ldots,v_n$, such that $L_n[\{v_1,v_2,\ldots,v_{n-1}\}]$ is a transitive tournament with $v_1\rightarrow v_2\rightarrow \cdots \rightarrow v_{n-1}$, and $v_n\rightarrow v_i(1\leq i\leq n-1)$ if $i$ is odd, $v_n\leftarrow v_i(1\leq i\leq n-1)$ otherwise \cite{Dfive}. Clearly, $L_2$ is a transitive tournament, $L_4$ is a diamond. If a tournament $T$ is isomorphic to $L_n$, we say $T$ is $L_n$. $L_2$, $L_4$ and $L_6$ are shown in Figure \ref{figLn}.
	
	\begin{figure}[h]
		\centering
		\includegraphics[scale=0.35]{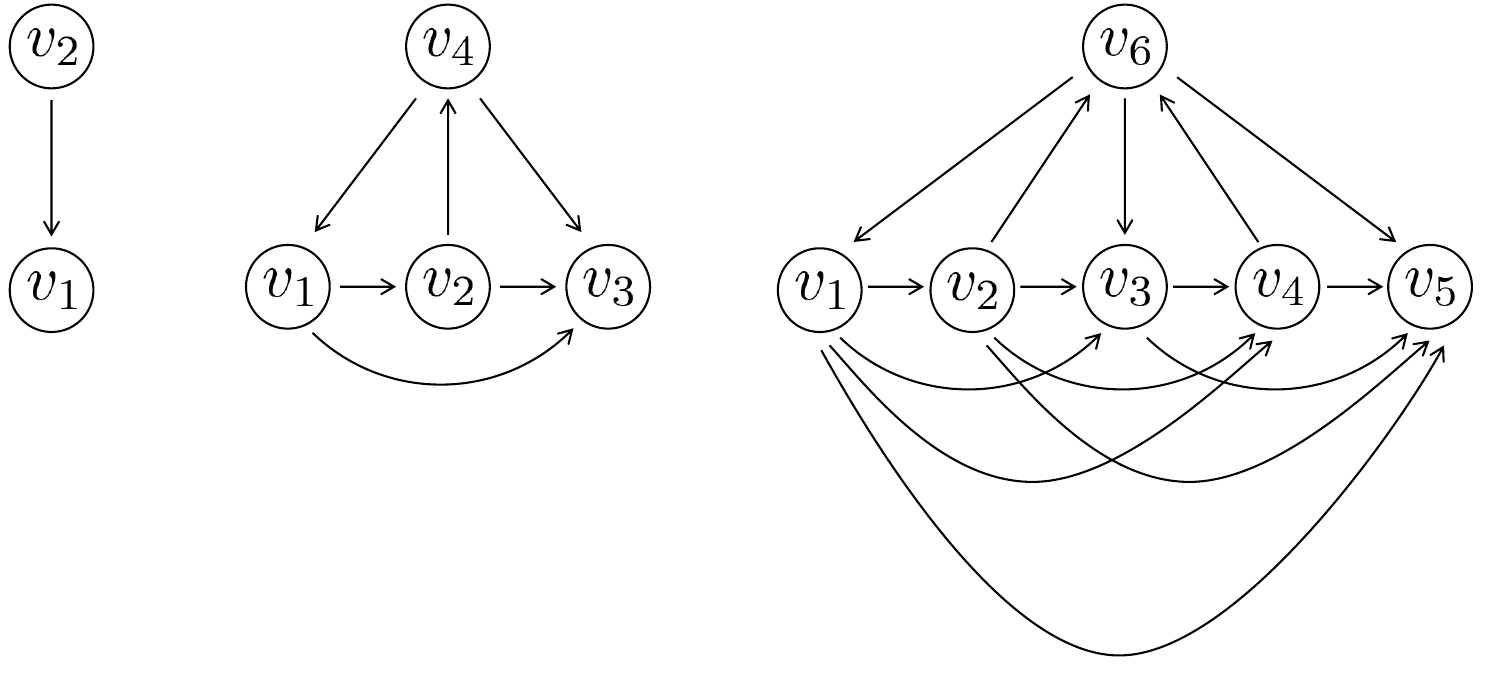}
		\caption{$L_2$, $L_4$ and $L_6$} \label{figLn}
	\end{figure}
	
	\begin{theorem}{\rm(\!\!\cite{Dfive})}\label{Lntheorem1}
		Let $n$ be a positive even integer. Then $\det(L_n)=(n-1)^2$, and $L_n\in \mathcal{D}_{n-1}\backslash \mathcal{D}_{n-3}$.
	\end{theorem}
	
	A tournament is a \textit{local order} if it contains no diamonds \cite{Localorder}. A tournament is a local order if and only if it is switching equivalent to a transitive tournament \cite{TRANSI}. A $5$-tournament contains exactly 0 or 2 diamonds. Based on these facts, the authors in \cite{DTHREE} characterized the sets $\mathcal{D}_1$ and $\mathcal{D}_{3}$ as follows.
	
	\begin{theorem}{\rm(\!\!\cite{DTHREE})}\label{D1}
		Let $T$ be a tournament. Then the following assertions are equivalent:
		\item{\rm(i)} $T\in \mathcal{D}_1$.
		\item{\rm(ii)} $T$ is switching equivalent to a transitive tournament.
		\item{\rm(iii)} $T$ contains no diamonds.
	\end{theorem}
	
	\begin{theorem}{\rm(\!\!\cite{DTHREE})}\label{Dthree}
		Let $T$ be a tournament. Then the following assertions are equivalent:
		\item{\rm(i)} $T\in \mathcal{D}_3$.
		\item{\rm(ii)} $T$ is switching equivalent to a transitive tournament or a transitive blowup of a diamond.
		\item{\rm(iii)} All the $6$-subtournaments of $T$ are in $\mathcal{D}_3$.
	\end{theorem}
	
	The authors \cite{Dfive} characterized the set $\mathcal{D}_5$, then $\mathcal{D}_1$, $\mathcal{D}_3$ in terms of $L_2$, $L_4$ as follows.
	
	\begin{theorem}{\rm(\!\!\cite{Dfive})}\label{D5character}
		Let $ T $ be an $n$-tournament $(n\geq 2)$. Then we have 
		\item{\rm(i)} $T\in \mathcal{D}_1$ if and only if $T$ is switching equivalent to a transitive blowup of $L_2$.
		\item{\rm(ii)} $T\in \mathcal{D}_3\backslash \mathcal{D}_1$ if and only if $T$ is switching equivalent to a transitive blowup of $L_4$.
		\item{\rm(iii)} $T\in \mathcal{D}_3$ if and only if $T$ is switching equivalent to a transitive blowup of $L_k$, where $k\in\{2,4\}$.
		\item{\rm(iv)} $T\in \mathcal{D}_5\backslash \mathcal{D}_3$ if and only if $T$ is switching equivalent to a transitive blowup of $L_6$.
		\item{\rm(v)} $T\in \mathcal{D}_5$ if and only if $T$ is switching equivalent to a transitive blowup of $L_k$, where $k\in\{2,4,6\}$.
	\end{theorem}
	
		Theorem \ref{Lntheorem1} shows that $\mathcal{D}_k\backslash\mathcal{D}_{k-2}$ is not an empty set, and a natural question arised from Theorem \ref{D5character} is that whether a tournament $T\in \mathcal{D}_k\backslash \mathcal{D}_{k-2}$ if and only if $T$ is switching equivalent to a transitive blowup of $L_{k+1}$ for $k\geq 7$. The following example tells us the answer is no. 
		
	\begin{example}{\rm(\!\!\cite{Dfive})}
		Let $T'$ be a $6$-tournament with  
		$
			\scriptsize{S_{T'}= \left[\begin{array}{rrrrrr}
					0 & 1 & 1 & 1 & -1 & -1 \\
					-1 & 0 & 1 & 1 & -1 & 1 \\
					-1 & -1 & 0 & 1 & 1 & 1 \\
					-1 & -1 & -1 & 0 & -1 & -1 \\
					1 & 1 & -1 & 1 & 0 & 1 \\
					1 & -1 & -1 & 1 & -1 & 0
				\end{array}\right]}.\label{exampleT}
		$	Clearly, $T' \in \mathcal{D}_{7} \backslash \mathcal{D}_{5}$,   and $T'$ can not be switching equivalent to a transitive blowup of $L_8$. 
	\end{example}	

	Based on Theorem \ref{Lntheorem1}, Theorem \ref{D5character} and Example 1.7, Zeng and You in \cite{Dfive} proposed the following question for further study.
	
	\begin{question}{\rm(\!\!\cite{Dfive})}\label{queanyDktwo}
		Let $k\text{ }(\geq 7)$ be odd. What is the necessary and sufficient condition for a tournament $T \in \mathcal{D}_{k} \backslash \mathcal{D}_{k-2}$ to be switching equivalent to a transitive blowup of $L_{k+1}$?
	\end{question}

	If $T\in \mathcal{D}_k\backslash \mathcal{D}_{k-2}$, and adding any vertex that does not conform to the structure of $T$ results in a new tournament $T'=T+\{u\}$ with a larger principal minor (that is, $T'\notin \mathcal{D}_k$), then $T$ can be considered to be in some kind of ``critical" state in this sense, which also means it possesses a speical extremal determinant property. Inspired by this, we study a special class of tournaments, which we define as \textit{CR tournaments}.
	 
	To solve the problem above and to further explore the structural properties of tournaments in $\mathcal{D}_{k}$, in this paper, we define CR tournaments, \textit{strong CR tournaments} and \textit{basic tournaments} (their definitions will be provided in the next section), show some conclusions on these tournaments, and prove that the properties of being ``a CR tournament'', ``a strong CR tournament'' and ``a basic tournament'' are invariants under switching operation (see Section \ref{sec-invariant}). Our main results (see Sections \ref{sec-proofoftheorem}, \ref{sec-allLn} and \ref{sec-solutionquestion}) are the following:
	\begin{itemize}
		\item In Section \ref{sec-proofoftheorem}, we establish a theorem on basic strong CR tournaments (see Theorem \ref{maintheorembasic}), which shows a relationship, in terms of $\mathcal{D}_{k} \backslash \mathcal{D}_{k-2}$, between a basic strong CR tournament $H$ and these tournaments containing a subtournament which is switching isomorphic to $H$.
		\item In Section \ref{sec-allLn}, we show that $L_n$ is a basic strong CR tournament for even $n \geq 4$ by using a specialized technique (see Theorem \ref{evenLnbasicCR} and Subsection \ref{subsecLn-pfeven}), and further deduce that all $L_n$ are strong CR tournaments (see Theorem \ref{allLnstrongCR}). Specifically, we introduce a class of matrices, denoted by $Z$-matrices (see Subsection \ref{subsecLn-table}), and complete the proof of Theorem \ref{evenLnbasicCR} by using $Z$-matrices. The proof presented in Subsection \ref{subsecLn-pfeven} is the most technical part of this paper.
		\item In Section \ref{sec-solutionquestion}, based on the results presented in Sections \ref{sec-proofoftheorem} and \ref{sec-allLn}, we give a complete solution to Question \ref{queanyDktwo} (see Theorem \ref{solutionque}), and propose some questions for further research.
	\end{itemize}
	
	\section{CR tournaments and basic tournaments}\label{sec-CR}
	\hspace{1.5em}Let $T$ be a tournament, $u_1,u_2\in V(T)$. We write $\theta_T(u_1,u_2)=1$ and $\theta_T(u_2,u_1)=-1$ if $u_1\rightarrow u_2$ in $T$.
	
	\begin{definition}\label{defcoorre}
		Let $T$ be a tournament, $u_1,u_2\in V(T)$. 
		\item{\rm(i)} When $|V(T)|=2$, $u_1$ and $u_2$ are called covertices in $T$ and revertices in $T$.
		\item{\rm(ii)} When $|V(T)|\geq 3$, $u_1$ and $u_2$ are called covertices in $T$ if $\theta_T(u_1,v)=\theta_T(u_2,v)$ for any $v\in V(T)\backslash\{u_1,u_2\}$; $u_1$ and $u_2$ are called revertices in $T$ if $\theta_T(u_1,v)=-\theta_T(u_2,v)$ for any $v\in V(T)\backslash\{u_1,u_2\}$.
		
		Furthermore, if $u_1$ and $u_2$ are covertices or revertices in $T$, we say $u_1$ and $u_2$ are CR-associated vertices in $T$.
	\end{definition}

		The following proposition and corollaries can be obtained directly from Definition \ref{defcoorre}, so we omit the proofs.
	\begin{proposition}\label{jugcoreone}
		Let $T$ be a tournament of order $n\geq 3$, $\{u_1,u_2\}\subset V(T)$. Then 
		\item{\rm(i)} $u_1$ and $u_2$ are covertices in $T$ if and only if $\theta_T(u_1,v)\cdot\theta_T(u_2,v)=1$ for any $v\in V(T) \backslash \{u_1,u_2\}$.
		\item{\rm(ii)} $u_1$ and $u_2$ are revertices in $T$ if and only if $\theta_T(u_1,v)\cdot\theta_T(u_2,v)=-1$ for any $v\in V(T) \backslash \{u_1,u_2\}$.
	\end{proposition}
	
	\begin{corollary}\label{jugcoretwo}
		Let $T$ be a tournament of order $n\geq 3$, $\{u_1,u_2\}\in V(T)$. Then $u_1$ and $u_2$ are CR-associated vertices in $T$ if and only if $(\theta_T(u_1,v_1)\cdot\theta_T(u_2,v_1))\cdot (\theta_T(u_1,v_2)\cdot\theta_T(u_2,v_2))=1$ for any $\{v_1,v_2\}\subseteq V(T)\backslash \{u_1,u_2\}$.
	\end{corollary}
	
	\begin{corollary}\label{subcore}
		Let $T$ be a tournament, $T_s$ be a subtournament of $T$, $\{u_1,u_2\}\subseteq V(T_s)$. Then we have
		\item{\rm(i)} $u_1$ and $u_2$ are covertices in $T_s$ if  $u_1$ and $u_2$ are covertices in $T$.
		\item{\rm(ii)} $u_1$ and $u_2$ are revertices in $T_s$ if  $u_1$ and $u_2$ are revertices in $T$.
	\end{corollary}
	
	We note that, in Corollary \ref{jugcoretwo}, it is allowed that $v_1=v_2$.
	
		Let $T$ be an $n$-tournament, $u\notin V(T)=\{v_1,\ldots,v_n\}$, $\sigma=(r_1,\ldots,r_n)$ be a dominating relation between $u$ and $V(T)$, where $r_i=1$ if $u \rightarrow v_i$ and $r_i=-1$ otherwise. 	Then there is an $(n+1)$-tournament generated by $T$ and $u$ with $\sigma$, which we denote by $T(u, \sigma)$.

	\begin{definition}\label{defcrvertex}
		Let $T$ be a tournament, $u\notin V(T)$, $T(u,\sigma)$ be an $(n+1)$-tournament generated by $T$ and $u$ with a dominating relation $\sigma$. We call $u$ a CR vertex for $T$ with $\sigma$ if there exists a vertex $v\in V(T)$ such that $u$ and $v$ are CR-associated vertices in $T(u,\sigma)$, and call $\sigma$ a CR dominating relation between $u$ and $V(T)$; we call $u$ a non-CR vertex for $T$ with $\sigma$ if $u$ is not a CR vertex for $T$ in $T(u,\sigma)$, and call $\sigma$ a non-CR dominating relation between $u$ and $V(T)$.
	\end{definition}
	
	The following lemma shows a key property of CR vertices. The proof will be given in the next section.
	
	\begin{lemma}\label{lemcrdet}
		Let $T\in \mathcal{D}_{k} \backslash \mathcal{D}_{k-2}$, $u \notin V(T)$,  $\sigma$ be a dominating relation between $u$ and $V(T)$, and $u$ be a CR vertex for $T$ with $\sigma$. Then $T(u,\sigma)\in \mathcal{D}_{k}\backslash \mathcal{D}_{k-2}$.
	\end{lemma}
	
	Clearly, $u$ is a non-CR vertex for $T$ with $\sigma$ if there exist no vertex $v\in V(T)$ such that $u$ and $v$ are CR-associated vertices in $T(u,\sigma)$. Among all the $2^n$ possible dominating relations between $u$ and the vertices in $V(T)$, there must exist a dominating relation $\sigma$ such that $u$ is a CR vertex for $T$ with $\sigma$. However, it is possible that there does not exist a dominating relation $\sigma$ between $u$ and the vertices in $V(T)$ such that $u$ is a non-CR vertex for $T$ with $\sigma$. In fact, we can show the following result immediately.
\begin{proposition}\label{nonCRvertex}
	Let $T$ be an $n$-tournament, $u\notin V(T)$. 
	\item{\rm(i)} If $T$ is a $1$-tournament, a $2$-tournament or a diamond, then for any dominating relation $\sigma$ between $u$ and $V(T)$, $u$ is a CR vertex for $T$ with $\sigma$.
	\item{\rm(ii)} If $n \geq 3$ and $T$ is not a diamond, then there exists a dominating relation $\sigma$ between $u$ and $V(T)$ such that $u$ is a non-CR vertex for $T$ with $\sigma$.
\end{proposition}

\begin{proof}
	 By direct checking, {\rm (i)} holds. Now we prove {\rm (ii)} holds by the following two cases.
	 
	 \textbf{Case 1}: $3 \leq n \leq 4$.
	 
	  In this case, by direct checking, if $T$ is not a diamond, then there exists a dominating relation $\sigma$ between $u$ and $V(T)$ such that $u$ is a non-CR vertex for $T$ with $\sigma$.
	
	\textbf{Case 2}: $n \geq 5$.
	
		Since $|V(T)|=n$, there are $2^n$ different dominating relations between $u$ and $V(T)$. For a vertex $v \in V(T)$, there are $4$ different dominating relations such that $u$ and $v$ are CR-associated vertices in $T(u,\sigma)$, and thus there are at most $4n$ different dominating relations such that $u$ is a CR vertex for $T$ with $\sigma$. When $n> 4$, we have $2^n-4n>0$. It follows that there exists a dominating relation $\sigma$ such that there exist no vertex $v\in V(T)$ satisfying the condition that $u$ and $v$ are CR-associated vertices in $T(u,\sigma)$, that is, $u$ is a non-CR vertex for $T$ with $\sigma$.
\end{proof}

	 Based on Proposition \ref{nonCRvertex}, now we give the definition of CR tournaments.
	
	\begin{definition}\label{defCR}
		Let $T\in \mathcal{D}_k\backslash \mathcal{D}_{k-2}$, $u \notin V(T)$.
		
		\item{\rm(i)} If $T$ is a $1$-tournament, a $2$-tournament or a diamond, we call $T$ a CR tournament, and call $T$ a trivial CR tournament.
		
		\item{\rm(ii)} If $T$ is not a trivial CR tournament, and for any dominating relation $\sigma$ between $u$ and $V(T)$ such that $u$ is a non-CR vertex for $T$ with $\sigma$, $T(u,\sigma)\notin \mathcal{D}_{k}$ holds, we call $T$ a CR tournament. 
	\end{definition}
	
	By Definition \ref{defCR}, if $T\in \mathcal{D}_{k}\backslash \mathcal{D}_{k-2}$ is a CR tournament and $T(u,\sigma)\in \mathcal{D}_{k}\backslash \mathcal{D}_{k-2}$, then $u$ must be a CR vertex for $T$ with $\sigma$. Moreover, if $u$ is a CR vertex for $T$ with $\sigma$, then $T(u,\sigma) \in \mathcal{D}_{k}\backslash \mathcal{D}_{k-2}$ by Lemma \ref{lemcrdet}. Therefore, an equivalent statement of Definition \ref{defCR} is as follows. 
	
	\begin{definition}\label{defCR2}
		Let $T\in \mathcal{D}_{k} \backslash \mathcal{D}_{k-2}$, $u \notin V(T)$, $\sigma$ be a dominating relation between $u$ and $V(T)$. If $T$ satisfies the condition that $T(u,\sigma)\in \mathcal{D}_{k} \backslash \mathcal{D}_{k-2}$ if and only if $u$ is a CR vertex for $T$ with $\sigma$, then $T$ is a CR tournament.
	\end{definition}
	
	In fact, if $T\in \mathcal{D}_k\backslash \mathcal{D}_{k-2}$ is a CR tournament and $u$ is a non-CR vertex for $T$ with $\sigma$, then $T(u,\sigma)$ contains at least a subtournament $T_{sub}(u,\sigma)$ such that $\det(T_{sub}(u,\sigma))>k^2$ (and $u$ must be contained in $V(T_{sub}(u,\sigma))$), which implies it will generate some new subtournaments with larger determinant by adding a non-CR vertex to $T$.
	
	Now, as an example, we show any $3$-tournament is a CR tournament.
	
	\begin{proposition}\label{ex3cycle}
		A $3$-tournament is a CR tournament.
	\end{proposition}
	\begin{proof}
		Let $T$ be a $3$-tournament with $V(T)=\{v_1,v_2,v_3\}$, $u\notin V(T)$, $T(u,\sigma)$ be the tournament generated by $T$ and $u$ with a dominating relation $\sigma=(r_1,r_2,r_3)$, where $r_i=1$ if $u\rightarrow v_i$ and $r_i=-1$ otherwise. 
		
		It is clear that $T\in \mathcal{D}_1\backslash \mathcal{D}_{-1}$, and $T$ is a $3$-cycle if $T$ is not a $3$-transitive tournament. If $T$ is a $3$-transitive tournament (resp. a $3$-cycle), then it is easy to check that $u$ is a non-CR vertex for $T$ with $\sigma$ if and only if $\sigma \in \{(1,-1,1), (-1,1,-1)\}$ (resp. $\sigma\in \{(1,1,1), (-1,-1,-1)\}$). For these two dominating relations, we have $\det(T(u,\sigma))=9$ and $T(u,\sigma)\notin \mathcal{D}_1$, thus $T$ is a CR tournament.
	\end{proof}
	
	If a tournament $T$ have two vertices $u_1$ and $u_2$ are CR-associated vertices in $T$, then there exists a tournament $H$ (for example, $H=T[V(T) \backslash \{u_1\}]$) such that $T$ is switching equivalent to a $1$-transitive blowup of $H$. Hence a tournament in which there does not exist two vertices $u_1$ and $u_2$ such that $u_1$ and $u_2$ are CR-associated vertices can be considered as a tournament having ``basic structure''. Now we give the definition of basic tournaments.
	
	\begin{definition}\label{defbasic}
		A tournament $T$ of order $n\geq 4$ is a basic tournament if there exist no two vertices $\{u_1,u_2\}\subset V(T)$ such that $u_1$ and $u_2$ are CR-associated vertices in $T$.
	\end{definition}
	
	Now we define \textit{strong CR tournaments}, which is a core concept of this paper.
	
	\begin{definition}\label{defstrongCR}
		A CR tournament $T$ is a strong CR tournament if every $1$-transitive blowup of $T$ is also a CR tournament.
	\end{definition}
	
	If a basic tournament $T$ is a CR tournament, we say $T$ is a basic CR tournament. If a basic tournament $T$ is a strong CR tournament, we say $T$ is a basic strong CR tournament.
	
	For small odd values of $k$, it is computationally feasible to directly verify whether $L_{k+1}$ is a basic strong CR tournament. For example, by direct checking, we have the following. 
	
		\begin{proposition}\label{L246}
		$L_n$ is a strong CR tournament for $n\in \{2,4,6\}$. In particular, $L_n$ is a basic strong CR tournament for $n\in \{4,6\}$.
	\end{proposition}
	
	However, determining whether $L_{k+1}$ is a basic strong CR tournament for $k \geq 7$ is challenging. This question will be completely solved in Section \ref{sec-allLn}.

\section{Some properties under switching operation}\label{sec-invariant}

\hspace{1.5em}In this section, we show some properties under switching operation. In particular, we show that the relationship ``CR-associated'' between two vertices and the properties of being ``a CR tournament'', ``a basic tournament'' and ``a strong CR tournament'' are invariants under switching operation. To avoid any potential confusion for the reader, we clarify that in this paper, whenever $\mathcal{D}_{k}$ or $\mathcal{D}_{k} \backslash \mathcal{D}_{k-2}$ appears, the subscript $k$ denotes a positive odd integer.

	\begin{lemma}{\rm(\!\!\cite{DTHREE})}\label{blowupclass}
	Let $T$ be an $n$-tournament, $T(a_1,\ldots,a_n)$ be a transitive blowup of $T$. Then $T\in \mathcal{D}_k$ if and only if $T(a_1,\ldots,a_n)\in \mathcal{D}_k$. 
\end{lemma}

By Lemma \ref{blowupclass}, we have the following result immediately.

\begin{corollary}\label{crllblowupclass}
	Let $T$ be an $n$-tournament, $T(a_1,\ldots,a_n)$ be a transitive blowup of $T$. Then $T\in \mathcal{D}_k\backslash \mathcal{D}_{k-2}$ if and only if $T(a_1,\ldots,a_n)\in \mathcal{D}_k\backslash \mathcal{D}_{k-2}$.  
\end{corollary}

\begin{lemma}{\rm(\!\!\cite{TWOGRAPH})}\label{Minors}
	Let tournaments $T_1$ and $T_2$ with the same vertex set $V$ be switching equivalent.  Then $T_1[U]$ is switching equivalent to $T_2[U]$, and $\det(T_1[U])=\det(T_2[U])$ for any non-empty subset $U\subseteq V$.
\end{lemma}

	\begin{corollary}\label{switchDk}
	Let tournaments $T_1$ and $T_2$ be switching isomorphic. Then $T_1\in \mathcal{D}_k\backslash \mathcal{D}_{k-2}$ if and only if $T_2\in \mathcal{D}_k\backslash\mathcal{D}_{k-2}$. 
\end{corollary}

\begin{proof}
	We only need to show the case where $T_1$ and $T_1$ are switching equivalent. Let $i\in \{1,2\}$, $j \in \{1,2\}\backslash \{i\}$. If $T_i\in \mathcal{D}_k\backslash \mathcal{D}_{k-2}$, then $\det(T_i[U])\leq k^2$ for any non-empty subset $U\subseteq V$ and there exists some subset $X\subseteq V$ such that $\det(T_i[X])=k^2$. Therefore, by Lemma \ref{Minors}, $\det(T_j[U])=\det(T_i[U]) \leq k^2$ for any non-empty subset $U\subseteq\ V$ and $\det(T_j[X])=\det(T_i[X])=k^2$, which implies $T_j\in \mathcal{D}_k\backslash \mathcal{D}_{k-2}$. 
	
	Combining the cases of $i=1$ and $i=2$, we complete the proof.
\end{proof}

Now we prove Lemma \ref{lemcrdet}.

\begin{proof}[{\bf Proof of Lemma \ref{lemcrdet}:}]
	Since $u$ is a CR vertex for $T$ with $\sigma$, there exists $v\in V(T)$ such that $u$ and $v$ are CR-associated vertices in $T(u,\sigma)$. Let $W=\emptyset$ if $u$ and $v$ are covertices in $T(u,\sigma)$, $W=\{u\}$ if $u$ and $v$ are revertices in $T(u,\sigma)$, and $T^{*}(u,\sigma)$ be a switch of $T(u,\sigma)$ with respect to $W$. Then $T^{*}(u,\sigma)$ is a $1$-transitive blowup of $T$. By Corollaries \ref{crllblowupclass} and  \ref{switchDk}, we have $T^{*}(u,\sigma)\in \mathcal{D}_{k}\backslash \mathcal{D}_{k-2}$ and $T(u,\sigma)\in \mathcal{D}_{k}\backslash \mathcal{D}_{k-2}$. This completes the proof.
\end{proof}

	\begin{lemma}\label{lemtranblowswitcheq}
	Let tournaments $T_1$ and $T_2$ with the same vertex set $V=\{v_1,v_2,\ldots,v_n\}$ be switching equivalent, $T_1(a_1,a_2,\ldots,a_n)$ and $T_2(a_1,a_2,\ldots,a_n)$ be a transitive blowup of $T_1$ and $T_2$ with respect to the same tournaments $H_1,\ldots,H_n$, respectively. Then $T_1(a_1,a_2,\ldots,a_n)$ is switching equivalent to $T_2(a_1,a_2,\ldots,a_n)$.
\end{lemma}

\begin{proof}
	Suppose that $T_1$ is switching equivalent to $T_2$ with respect to $W$.
	
	If $W=\emptyset$, then $T_1=T_2$, and thus $T_1(a_1,a_2,\ldots,a_n)=T_2(a_1,a_2,\ldots,a_n)$.
	
	If $W=\{v_{i_1},v_{i_2},\ldots,v_{i_s}\}\neq \emptyset$, we take $W'=V(H_{i_1}) \cup V(H_{i_2}) \cup \cdots \cup V(H_{i_s})$. Then $T_1(a_1,a_2,\ldots,a_n)$ is switching equivalent to $T_2(a_1,a_2,\ldots,a_n)$ with respect to $W'$.
	
	Combining the above arguments, we complete the proof.
\end{proof}

The following lemma is crucial for the  proof of the subsequent theorem.

\begin{lemma}\label{switchnochgcore}
	Let $T$ be a tournament, $T'$ be a switch of $T$, $u_1$ and $u_2$ be CR-associated vertices in $T$. Then $u_1$ and $u_2$ are CR-associated vertices in $T'$.
\end{lemma}

\begin{proof}
	If $|V(T)|=|V(T')|=2$, then the result holds by Definition \ref{defcoorre}. Now we consider $|V(T)|\geq 3$.
	
	Suppose that $T'$ is a switch of $T$ with respect to the subset $W\subseteq V(T)$. For any $v\in V(T) \backslash \{u_1,u_2\}$ and $i \in \{1,2\}$, we have
	\begin{align}
		\theta_{T'}(u_i,v)=\begin{cases}
			\theta_{T}(u_i,v), & \text{if $\{u_i,v\}\subseteq W$ or $\{u_i,v\}\subseteq V(T)\backslash W$};\\
			-\theta_{T}(u_i,v), & \text{if $|\{u_i,v\} \cap W|=1$}.
		\end{cases} \label{gs4-1-1}
	\end{align}
	
	Then we complete the proof by the following two cases.
	
	\textbf{Case 1}: $\{u_1,u_2\}\subseteq W$ or $\{u_1,u_2\}\subseteq V(T)\backslash W$.
	
		In this case, by \eqref{gs4-1-1}, we have $\theta_{T'}(u_1,v)\cdot \theta_{T'}(u_2,v)=\theta_{T}(u_1,v)\cdot \theta_{T}(u_2,v)$ for any $v \in V(T)\backslash \{u_1,u_2\}$, and then by Proposition \ref{jugcoreone}, $u_1$ and $u_2$ are covertices (resp. revertices) in $T'$ if $u_1$ and $u_2$ are covertices (resp. revertices) in $T$.
	
	\textbf{Case 2}: $|\{u_1,u_2\} \cap W|=1$.
	
	In this case, by \eqref{gs4-1-1}, we have $\theta_{T'}(u_1,v)\cdot \theta_{T'}(u_2,v)=-\theta_{T}(u_1,v)\cdot \theta_{T}(u_2,v)$ for any $v \in V(T)\backslash \{u_1,u_2\}$, and then by Proposition \ref{jugcoreone}, $u_1$ and $u_2$ are revertices (resp. covertices) in $T'$ if $u_1$ and $u_2$ are covertices (resp. revertices) in $T$.
\end{proof}

\begin{corollary}\label{switchnonCR}
	Let $T$ be a tournament, $u$ be a non-CR vertex for $T$ with a dominating relation $\sigma$, $T(u,\sigma)$ be the tournament generated by $T$ and $u$ with $\sigma$, and $T'(u,\sigma)$ be a switch of $T(u,\sigma)$. Then $u$ is a non-CR vertex for $T'(u,\sigma)[V(T)]$.
\end{corollary}

\begin{proof}
	Suppose that $u$ is a CR vertex for $T'(u,\sigma)[V(T)]$, then there exists a vertex $v\in V(T)$ such that $u$ and $v$ are CR-associated vertices in $T'(u,\sigma)$, and thus $u$ and $v$ are also CR-associated vertices in $T(u,\sigma)$ by Theorem \ref{switchnochgcore}, a contradiction.
	
	Therefore, $u$ is also a non-CR vertex for $T'(u,\sigma)[V(T)]$.
\end{proof}

\begin{theorem}\label{thforswiso}
	Let $T_1$ and $T_2$ be tournaments such that $T_1$ is switching isomorphic to $T_2$. Then
	\item{\rm(i)} $T_1$ is a CR tournament if and only if $T_2$ is a CR tournament.
	\item{\rm(ii)} $T_1$ is a strong CR tournament if and only if $T_2$ is a strong CR tournament.
	\item{\rm(iii)} $T_1$ is a basic tournament if and only if $T_2$ is a basic tournament.
	\item{\rm(iv)} $T_1$ is a basic CR tournament if and only if $T_2$ is a basic CR tournament.
	\item{\rm(v)} $T_1$ is a basic strong CR tournament if and only if $T_2$ is a basic strong CR tournament.
\end{theorem}

\begin{proof}
	As (i), (ii) and (iii) implies (iv) and (v), we only need to show that (i), (ii), and (iii) hold when $T_1$ and $T_2$ are switching equivalent with the vertex set $V$. Without loss of generality, we only need to prove that if $T_1$ is a CR tournament (strong CR tournament, basic tournament), then $T_2$ is also a CR tournament (strong CR tournament, basic tournament).
	
	{\bf Proof of (i):} Assume that $T_1\in \mathcal{D}_{k} \backslash \mathcal{D}_{k-2}$. Since $T_2$ is a switch of $T_1$, we have $T_2\in \mathcal{D}_k\backslash \mathcal{D}_{k-2}$ by Corollary \ref{switchDk}.
	
	If $T_1$ is a trivial CR tournament, the result holds immediately by Definition \ref{defCR}. Now we suppose that $T_1$ is not a trivial CR tournament.
	
	Let $u$ be a non-CR vertex for $T_2$ with a dominating relation $\sigma$, and $T_2(u,\sigma)$ be the tournament generated by $T_2$ and $u$ with $\sigma$. Then there exists a switch of $T_2(u,\sigma)$, denoted by $T'_2(u,\sigma)$, such that $T'_2(u,\sigma)[V]=T_1$. Now we show $T_2(u,\sigma) \notin \mathcal{D}_k$.
	
	If $T_2(u,\sigma)\in \mathcal{D}_k$, then $T'_2(u,\sigma)\in \mathcal{D}_k$ by Corollary \ref{switchDk} and the fact that $T'_2(u,\sigma)$ is a switch of $T_2(u,\sigma)$. Now $T'_2(u,\sigma)$ is a tournament generated by $T_1$ and $u$ with the dominating relation determined by the structure of $T'_2(u,\sigma)$.  If $u$ is a non-CR vertex for $T_1$, then by {\rm(ii)} of Definition \ref{defCR} and the fact that $T_1$ is a CR tournament but not trivial, we have $T'_2(u,\sigma)\notin \mathcal{D}_{k}$, a contradiction. Therefore, $u$ is a CR vertex for $T_1$, which implies there exists a vertex $v\in V$ such that $u$ and $v$ are CR-associated vertices in $T'_2(u,\sigma)$. By Lemma \ref{switchnochgcore}, $u$ and $v$ are CR-associated vertices in $T_2(u,\sigma)$, which contradicts that $u$ is a non-CR vertex for $T_2$. Therefore, $T_2(u,\sigma) \notin \mathcal{D}_k$, it follows that $T_2$ is a CR tournament.
	
	{\bf Proof of (ii):} Let $\hat{T_2}$ be a $1$-transitive blowup of $T_2$. Then by Lemma \ref{lemtranblowswitcheq}, $\hat{T_2}$ is switching equivalent to a $1$-transitive blowup of $T_1$, denoted by $\hat{T_1}$. Since $T_1$ is a strong CR tournament, we have $\hat{T_1}$ is a CR tournament. Then $\hat{T_2}$ is a CR tournament by (i). Therefore, $T_2$ is a strong CR tournament by Definition \ref{defstrongCR}.
	
	{\bf Proof of (iii):} If $T_1$ is a basic tournament and $T_2$ is not a basic tournament, then there exist two vertices $\{v_1,v_2\}\subset V$ such that $v_1$ and $v_2$ are CR-associated vertices in $T_2$, and thus $v_1$ and $v_2$ are CR-associated vertices in $T_1$ by Lemma \ref{switchnochgcore}, which contradicts that $T_1$ is a basic tournament. Therefore, $T_2$ is also a basic tournament.
\end{proof}

Based on the above results, we now show some further properties of CR tournaments, strong CR tournaments and basic tournaments, .

\begin{proposition}\label{tranblowuptoCR}
	Let $T$ be an $n$-tournament with $V(T)=\{v_1,\ldots,v_n\}$, $\hat{T}=T(a_1,\ldots,a_n)$ be a transitive blowup of $T$ with $a_i=|V(H_i)|$ as in Definition \ref{DefBLWOUP}. If $\hat{T}$ is a CR tournament, then $T$ is a CR tournament.
\end{proposition}

\begin{proof}
	Assume that $T\in \mathcal{D}_k\backslash \mathcal{D}_{k-2}$ for some odd $k$. Then $\hat{T} \in \mathcal{D}_{k} \backslash \mathcal{D}_{k-2}$ by Corollary \ref{crllblowupclass}.
	
	If $\hat{T}$ is a trivial CR tournament, then it is easy to check that $T$ is a CR tournament by Definition \ref{defCR} and Proposition \ref{ex3cycle}. Now we consider the case that $\hat{T}$ is not a trivial CR tournament.
	
	If $T$ is a $1$-tournament, a $2$-tournament or a diamond, then $T$ is a CR tournament by Definition \ref{defCR}. If $T$ is not  a $1$-tournament, a $2$-tournament or a diamond, then there exist non-CR vertices for $T$.
	Let $u$ $(\notin V(T))$ be a non-CR vertex for $T$ with a dominating relation $\sigma$, $T(u,\sigma)$ be the tournament generated by $T$ and $u$ with $\sigma$, $\tilde{u}$ $(\notin V(\hat{T}))$ be a vertex, and $\hat{T}(\tilde{u},\tilde{\sigma})$ be the tournament generated by $\hat{T}$ and $\tilde{u}$ with the dominating relation $\tilde{\sigma}$ between $\tilde{u}$ and $V(\hat{T})$ such that $\{\tilde{u}\}\rightarrow V(H_i)$ if $u\rightarrow v_i$ and $\{\tilde{u}\}\leftarrow V(H_i)$ if $u\leftarrow v_i$. Clearly, $\hat{T}(\tilde{u},\tilde{\sigma})$ is a transitive blowup of $T(u,\sigma)$ such that $\hat{T}(\tilde{u},\tilde{\sigma})=T(u,\sigma)(a_1,a_2,\ldots,a_n,1)$.
	
	If there exists some $j \in \{1,2,\ldots,n\}$ such that $w\in V(H_j) \subseteq V(\hat{T})$ and $\tilde{u}$ are CR-associated vertices in $\hat{T}(\tilde{u},\tilde{\sigma})$, then it is not hard to see that $u$ and $v_j$ are CR-associated vertices in $T(u,\sigma)$, a contradiction. Hence there exist no such $w\in V(\hat{T})$ satisfying the condition that $\tilde{u}$ and $w$ are CR-associated vertices in $\hat{T}(\tilde{u},\tilde{\sigma})$, which implies $\tilde{u}$ is a non-CR vertex for $\hat{T}$ with $\tilde{\sigma}$ and thus $\hat{T}(\tilde{u},\tilde{\sigma})\notin \mathcal{D}_k$. Then by Lemma \ref{blowupclass}, we have $T(u,\sigma)\notin \mathcal{D}_k$, and thus $T$ is a CR tournament.
\end{proof}

Recall that a tournament $T$ is a strong CR tournament if $T$ is a CR tournament and all $1$-transitive blowups of $T$ are CR tournaments. By using Proposition \ref{tranblowuptoCR}, we have the following proposition, which provide another definition of strong CR tournaments. C ompared with Definition \ref{defstrongCR}, Proposition \ref{strongCR1} allows one to determine whether $T$ is a strong CR tournament by only examining whether every $1$-transitive blowup of $T$ is a CR tournament.

\begin{proposition}\label{strongCR1}
	A tournament $T$ is a strong CR tournament if and only if all $1$-transitive blowups of $T$ are CR tournaments.
\end{proposition}

\begin{proof}
	Necessity. If $T$ is a strong CR tournament, then all $1$-transitive blowups of $T$ are CR tournaments.
	
	Sufficiency. If all $1$-transitive blowups of $T$ are CR tournaments, then by Proposition \ref{tranblowuptoCR}, we have $T$ is a CR tournament. Furthermore, $T$ is a strong CR tournament.
\end{proof}

\begin{proposition}\label{lem4maintheorem}
	Let $T$ be a basic tournament and $\hat{T}$ be a transitive blowup of $T$, $u$ be a non-CR vertex for $\hat{T}$ with a non-CR dominating relation $\sigma$. Then $\hat{T}(u,\sigma)$ can not be switching equivalent to a transitive blowup of $T$.
\end{proposition}

\begin{proof}
	Assume that $|V(T)|=n$. Then $n\geq 4$ by \ref{defbasic}. Since $\hat{T}$ is a transitive blowup of $T$, there exists a subset $Z\subset V(\hat{T})$ such that $\hat{T}[Z]$ is isomorphic to $T$.
	
	Suppose, for the sake of contradiction, that $\hat{T}(u,\sigma)$ can be switching equivalent to a transitive blowup of $T$, denoted by $\hat{T_1}$. Then there exist positive integers $a_1,a_2,\ldots,a_n$ and subsets $X_1,X_2,\ldots,X_n \subset V(\hat{T_1})$ such that $|X_i|=a_i$ for all $i\in \{1,2,\ldots,n\}$, $\hat{T_1}[X_i]$ is transitive for all $i\in \{1,2,\ldots,n\}$, and $\hat{T_1}$ can be denoted by $T(\hat{T_1}[X_1],\hat{T_1}[X_2],\ldots,\hat{T_1}[X_n])$, or equivalently, $T(a_1,a_2,\ldots,a_n)$.
	
	Let $u\in X_i$, where $i \in \{1,\ldots,n\}$. We complete the proof by the following cases.
	
	\textbf{Case 1}: $a_i \geq 2$.
	
	Let $X_i=\{v_1,v_2,\ldots,v_{a_i}\}$ such that $v_1 \rightarrow v_2 \rightarrow \cdots \rightarrow v_{a_i}$, $u=v_s$ with $1\leq s \leq a_i$ and $w=\begin{cases}
		v_{s+1}, \, \, \quad if\ s\neq a_i; \\
		v_{a_i-1}, \quad if\ s=a_i.
	\end{cases}$ Then $u$ and $w$ are CR-associated vertices in $\hat{T_1}$, and thus $u$ and $w$ are CR-associated vertices in $\hat{T}(u,\sigma)$ by Lemma \ref{switchnochgcore}, which contradicts that $u$ is a non-CR vertex for $\hat{T}$ with $\sigma$.
	
	\textbf{Case 2}: $a_i=1$.
	
	By Lemma \ref{Minors}, $\hat{T}(u,\sigma)[Z]$ and $\hat{T_1}[Z]$ are switching equivalent, then $\hat{T}[Z]$ and $\hat{T_1}[Z]$ are switching equivalent by $\hat{T}[Z]=\hat{T}(u,\sigma)[Z]$. Since $\hat{T}[Z]$ is isomorphic to $T$ and $T$ is a basic tournament, $\hat{T}[Z]$ is a basic tournament. Furthermore, $\hat{T_1}[Z]$ is a basic tournament by Theorem \ref{thforswiso}.
	
	Without loss of generality, we assume that $u \in X_1$, and then $Z\subset X_2\cup X_3\cup \cdots \cup X_n$. Since $|Z|=n$, there exists $X_j$ such that $|Z \cap X_j|=t \geq 2$, where $2 \leq j \leq n$. Let $Z \cap X_j=\{w_1,w_2,\ldots,w_t\}$ such that $w_1 \rightarrow w_2 \rightarrow \cdots \rightarrow w_t$ in $\hat{T_1}$. Then $w_1$ and $w_2$ are CR-associated vertices in $\hat{T_1}[Z]$, which contradicts that $\hat{T_1}[Z]$ is a basic tournament.
	
	Therefore, $\hat{T}(u,\sigma)$ can not be switching equivalent to a transitive blowup of $T$.
\end{proof}

\begin{proposition}\label{basicnotD1}
	Let $T$ be a basic tournament. Then $T\notin \mathcal{D}_1\backslash \mathcal{D}_{-1}$.
\end{proposition}

\begin{proof}
	By Definition \ref{defbasic}, $|V(T)|\geq 4$. If $T\in \mathcal{D}_1\backslash \mathcal{D}_{-1}$, then by Theorem \ref{D5character}, $T$ is switching equivalent to a transitive blowup of $L_2$, denoted by $T'$. Since $T$ is a basic tournament, we have $T'$ is a basic tournament by Theorem \ref{thforswiso}. On the other hand, since $T'$ is a transitive blowup of $L_2$, there exists $\{v_1,v_2\}\subset V(T')=V(T)$ such that $v_1$ and $v_2$ are CR-associated vertices in $T'$, which implies $T'$ is not a basic tournament, a contradiction.  Therefore, $T\notin \mathcal{D}_1\backslash \mathcal{D}_{-1}$.
\end{proof}

\section{A main result on basic strong CR tournaments}\label{sec-proofoftheorem}
\hspace{1.5em}Let $H$ be a tournament. In this paper, we use $\xi(H)$ to denote the set of tournaments such that $T\in \xi(H)$ if and only if $T$ contains a subtournament which is switching isomorphic to $H$.

Let $T$ be an $n$-tournament, $u\in V(T)$ and $H$ be a subtournament of $T-u$ (i.e., $T-u=T[V(T)\backslash \{u\}]$). Then the dominating relation $\sigma$ between $u$ and $V(H)$ in $T$ is known and $H(u,\sigma)=T[V(H)\cup \{u\}]$, we usually denote $H(u,\sigma)$ by $H(u)$ for short, and simply say ``$u$ is a CR vertex (resp. non-CR vertex) for $H$'' if $u$ is a CR vertex (resp. non-CR vertex) for $H$ with $\sigma$ (omitting the specification of $\sigma$) in the following.

The following Theorem \ref{maintheorembasic} establishes a connection between a basic strong CR tournament $H$ and $\xi(H)$, which is our main theorem on CR tournaments. We will subsequently show how to use this theorem to get the same characterizations of $\mathcal{D}_3\backslash \mathcal{D}_1$ and $\mathcal{D}_5 \backslash \mathcal{D}_3$ as in Theorem \ref{D5character}, which implies Theorem \ref{maintheorembasic} maybe a useful tool for characterizing $\mathcal{D}_k$.

\begin{theorem}\label{maintheorembasic}
	Let $k$ $(\geq 3)$ be odd and $H\in \mathcal{D}_k\backslash \mathcal{D}_{k-2}$ be a basic tournament. Then the following assertions are equivalent:
	\item{\rm(i)} $H$ is a strong CR tournament.
	\item{\rm(ii)} All transitive blowups of $H$ are CR tournaments.
	\item{\rm(iii)} $T\in \xi(H) \cap (\mathcal{D}_k\backslash\mathcal{D}_{k-2})$ if and only if $T$ is switching equivalent to a transitive blowup of $H$.
\end{theorem}

To prove Theorem \ref{maintheorembasic}, we need the following lemmas.

\begin{lemma}\label{lem1maintheorem}
Let $T$ be a tournament with $V(T)=\{u_1,u_2,v_1,v_2,\ldots,v_n\}$, $T_s=T[\{v_1,v_2,\ldots,v_n\}]$ be a basic tournament, and there exists $i,j \in \{1,\ldots,n\}$ with $i \neq j$ such that $u_1$ and $v_i$ are covertices in $T[V(T_s)\cup \{u_1\}]$, $u_2$ and $v_j$ are covertices in $T[V(T_s)\cup \{u_2\}]$, $\theta_T(u_1,u_2)\cdot \theta_T(v_i,v_j)=-1$. Then $T$ is a basic tournament.
\end{lemma}

\begin{proof}
	Let $T^{(k)}_s=T[V(T_s)\cup \{u_k\}]$ for $k \in \{1,2\}$. Without loss of generality, we assume that $i=1$ and $j=2$. Since $T_s$ is a basic tournament, we have $n \geq 4$ by Definition \ref{defbasic}.
	
	Suppose, for the sake of contradiction, that $T$ is not a basic tournament. Then there exist two vertices $w_1,w_2$ of $T$ such that $w_1$ and $w_2$ are CR-associated vertices in $T$.
	
	\textbf{Case 1}: $\{w_1,w_2\}\subset \{v_1,v_2,\ldots,v_n\}$.
	
	In this case, $w_1$ and $w_2$ are CR-associated vertices in $T_s$ by Corollary \ref{subcore}, which contradicts that $T_s$ is a basic tournament. 
	
	\textbf{Case 2}: $w_q\in \{u_1,u_2\}$, $w_p\in \{v_3,v_4,\ldots,v_n\}$, where $q\in \{1,2\}$, $p\in \{1,2\}\backslash \{q\}$.
	
	Without loss of generality, we only need to consider $q=1$. 
	
	Let $w_1=u_k$ and $X=\{u_k\}\cup \{v_1,v_2,\ldots,v_n\}\backslash\{v_k\}$, where $k \in \{1,2\}$. Then $T[X]$ is isomorphic to $T_s$ since $u_k$ and $v_k$ are covertices in $T^{(k)}_s$, and thus $T[X]$ is a basic tournament. But $w_1$ and $w_2$ are CR-associated vertices in $T[X]$ by Corollary \ref{subcore}, there is a contradiction.
	
	\textbf{Case 3}: $\{w_1,w_2\}\subset \{u_1,u_2,v_1,v_2\}$.
	
	\textbf{Subcase 3.1}: $\{w_1,w_2\}=\{v_1,v_2\}$.
	
	In this subcase, $v_1$ and $v_2$ are CR-associated vertices in $T_s$ by Corollary \ref{subcore}, which contradicts that $T_s$ is a basic tournament.
	
	\textbf{Subcase 3.2}: $w_q\in \{u_1,u_2\}$, $w_p\in \{v_1,v_2\}$, where $q\in \{1,2\}$, $p\in \{1,2\}\backslash \{q\}$.
	
	Without loss of generality, we can assume that $q=1$ and $w_1=u_1$.
	
	If $w_2=v_1$, then $\theta_T(w_1,v_2)\cdot \theta_T(w_2,v_2)=\theta_T(u_1,v_2)\cdot \theta_T(v_1,v_2)=1$ since $u_1$ and $v_1$ are covertices in $T^{(1)}_s$, and thus $w_1$ and $w_2$ are covertices in $T$ by the assumption that $w_1$ and $w_2$ are CR-associated vertices in $T$, which implies $\theta_T(w_1,u_2)\cdot \theta_T(w_2,u_2)=1$. On the other hand, since $u_2$ and $v_2$ are covertices in $T^{(2)}_s$, we have $\theta_T(v_1,u_2)=\theta_T(v_1,v_2)$, and thus $\theta_T(w_1,u_2)\cdot \theta_T(w_2,u_2)=\theta_T(u_1,u_2)\cdot \theta_T(v_1,u_2)=\theta_T(u_1,u_2)\cdot \theta_T(v_1,v_2)=-1$, a contradiction.
	
	If $w_2=v_2$, then $u_1$ and $v_2$ are CR-associated vertices in $T$, it follows that $u_1$ and $v_2$ are CR-associated vertices in $T^{(1)}_s$ by Corollary \ref{subcore}. Since $u_1$ and $v_1$ are covertices in $T^{(1)}_s$, then for any $\{v_{i_1},v_{i_2}\}\in V(T_s)\backslash \{v_1,v_2\}$, we have
	$$
	\begin{aligned}
		&(\theta_{T_s}(v_1,v_{i_1})\cdot\theta_{T_s}(v_2,v_{i_1}))\cdot (\theta_{T_s}(v_1,v_{i_2})\cdot\theta_{T_s}(v_2,v_{i_2}))\\
		=&(\theta_{T^{(1)}_s}(u_1,v_{i_1})\cdot\theta_{T^{(1)}_s}(v_2,v_{i_1}))\cdot (\theta_{T^{(1)}_s}(u_1,v_{i_2})\cdot\theta_{T^{(1)}_s}(v_2,v_{i_2}))\\
		=&1.
	\end{aligned}
	$$
	Therefore, $v_1$ and $v_2$ are CR-associated vertices in $T_s$ by Corollary \ref{jugcoretwo}, which contradicts that $T_s$ is a basic tournament. 
	
	\textbf{Subcase 3.3}: $\{w_1,w_2\}=\{u_1,u_2\}$.
	
	By using Corollary \ref{jugcoretwo}, we have $(\theta_T(u_1,v_{i_1})\cdot\theta_T(u_2,v_{i_1}))\cdot (\theta_T(u_1,v_{i_2})\cdot\theta_T(u_2,v_{i_2}))=1$ for any $\{v_{i_1},v_{i_2}\}\in V(T)\backslash \{u_1,u_2\}$. Since $u_1$ and $v_1$ are covertices in $T^{(1)}_s$, $u_2$ and $v_2$ are covertices in $T^{(2)}_s$, we have
	$$
	\begin{aligned}
		&(\theta_{T_s}(v_1,v_{i_1})\cdot\theta_{T_s}(v_2,v_{i_1}))\cdot (\theta_{T_s}(v_1,v_{i_2})\cdot\theta_{T_s}(v_2,v_{i_2}))\\
		=&(\theta_{T^{(1)}_s}(u_1,v_{i_1})\cdot\theta_{T^{(2)}_s}(u_2,v_{i_1}))\cdot (\theta_{T^{(1)}_s}(u_1,v_{i_2})\cdot\theta_{T^{(2)}_s}(u_2,v_{i_2}))\\
		=&(\theta_{T}(u_1,v_{i_1})\cdot\theta_{T}(u_2,v_{i_1}))\cdot (\theta_{T}(u_1,v_{i_2})\cdot\theta_{T}(u_2,v_{i_2}))\\
		=&1
	\end{aligned}
	$$
	for any $\{v_{i_1},v_{i_2}\}\in V(T_s)\backslash \{v_1,v_2\}$. Thus $v_1$ and $v_2$ are CR-associated vertices in $T_s$ by Corollary \ref{jugcoretwo}, which contradicts that $T_s$ is a basic tournament. 
	
	Combining the above arguments, $T$ is a basic tournament.
\end{proof}

\begin{lemma}\label{lem2maintheorem}
	Let $T$ be a basic tournament with $V(T)=\{v_1,v_2,\ldots,v_n\}$, $T^{*}$ be a blowup of $T$ with respect to $H_1,H_2,\ldots,H_n$, where $H_i$ is a 3-cycle and $|V(H_j)|=1$ for each $j\in \{1,2,\ldots,n\}\backslash \{i\}$. Then $T^{*}$ is a basic tournament.
\end{lemma}

\begin{proof}
	Without loss of generality, we assume that $H_1$ is a 3-cycle. Let $V(H_1)=\{w_1,w_2,w_3\}$ such that $w_1\rightarrow w_2$, $w_2\rightarrow w_3$ and $w_3\rightarrow w_1$, and $V(H_j)=\{u_j\}$ for each $j\in \{2,3,\ldots,n\}$. Then $V(T^{*})=\{w_1,w_2,w_3,u_2,\ldots,u_n\}$. Since $T$ is a basic tournament, we have $n\geq 4$ and $T^{*}[\{w_i,u_2,\ldots,u_n\}]$ is a basic tournament by the fact that $T^{*}[\{w_i,u_2,\ldots,u_n\}]$ is isomorphic to $T$ for $i \in \{1,2,3\}$.
	
	If $u_{j_1}$ and $u_{j_2}$ are CR-associated vertices in $T^{*}$ for $2 \leq j_1,j_2 \leq n$ with $j_1 \neq j_2$, then $u_{j_1}$ and $u_{j_2}$ are CR-associated vertices in $T^{*}[\{w_1,u_2,\ldots,u_n\}]$ by Corollary \ref{subcore}, which implies a contradiction since $T^{*}[\{w_1,u_2,\ldots,u_n\}]$ is a basic tournament.
	
	If $w_i$ and $u_j$ are CR-associated vertices in $T^{*}$ for $i \in \{1,2,3\}$ and $j \in \{2,3,\ldots,n\}$, then  $w_i$ and $u_j$ are also CR-associated vertices in $T^{*}[\{w_i,u_2,\ldots,u_n\}]$, a contradiction. 
	
	If $w_i$ and $w_j$ are CR-associated vertices in $T^{*}$ for $1 \leq i,j \leq 3$ with $i \neq j$, then $w_i$ and $w_j$ are CR-associated vertices in $T^{*}[\{w_1,w_2,w_3,u_2\}]$ by Corollary \ref{subcore}. But $T^{*}[\{w_1,w_2,w_3,u_2\}]$ is a diamond, and a diamond is a basic tournament by Definition \ref{defbasic}, a contradiction.
	
	Combining the above arguments, there exist no two vertices of $T^{*}$ such that the two vertices are CR-associated vertices in $T^{*}$. Therefore, $T^{*}$ is a basic tournament.
\end{proof}
\begin{lemma}\label{lem3maintheorem}
	Let $T$ be a basic $n$-tournament, $u\notin V(T)=\{v_1,v_2,\ldots,v_n\}$, $\sigma$ be a dominating relation between $u$ and $V(T)$, $T(u,\sigma)$ be a tournament generated by $T$ and $u$ with $\sigma$. If there exists $v_i \in V(T)$ such that $u$ and $v_i$ are CR-associated vertices in $T(u,\sigma)$, then $u$ and $v_j$ are not CR-associated vertices in $T(u,\sigma)$ for $j \in \{1,2,\ldots,n\} \backslash \{i\}$.
\end{lemma}

\begin{proof}
	Without loss of generality, we assume that $u$ and $v_1$ are CR-associated vertices in $T(u,\sigma)$. Suppose, for the sake of contradiction, that there exists $j\in \{2,\ldots,n\}$ such that $u$ and $v_j$ are CR-associated vertices in $T(u,\sigma)$. Since $T$ is a subtournament of $T(u,\sigma)$, then for any $k \in \{2,3,\ldots,n\} \backslash \{j\}$, we have
	\begin{align*}
		\theta_{T}(v_1,v_k) \cdot \theta_{T}(v_j,v_k)&=\theta_{T(u,\sigma)}(v_1,v_k) \cdot \theta_{T(u,\sigma)}(v_j,v_k) \\
		&=\theta_{T(u,\sigma)}(v_1,v_k) \cdot \theta_{T(u,\sigma)}(v_j,v_k) \cdot (\theta_{T(u,\sigma)}(u,v_k))^2\\
		&=(\theta_{T(u,\sigma)}(v_1,v_k)\cdot \theta_{T(u,\sigma)}(u,v_k)) \cdot (\theta_{T(u,\sigma)}(v_j,v_k) \cdot \theta_{T(u,\sigma)}(u,v_k)).
	\end{align*}
	
	By Proposition \ref{jugcoreone}, $\theta_{T(u,\sigma)}(v_1,v_k)\cdot \theta_{T(u,\sigma)}(u,v_k) = \alpha_1$ for any $k \in \{2,3,\ldots,n\} \backslash \{j\}$, where $\alpha_1 \in \{1,-1\}$ is a constant. Similarly, $\theta_{T(u,\sigma)}(v_j,v_k)\cdot \theta_{T(u,\sigma)}(u,v_k)= \alpha_2$ for any $k \in \{2,3,\ldots,n\} \backslash \{j\}$, where $\alpha_2 \in \{1,-1\}$ is a constant. Then $\theta_{T}(v_1,v_k) \cdot \theta_{T}(v_j,v_k) = \alpha_1 \alpha_2$ for any $k \in \{2,3,\ldots,n\} \backslash \{j\}$, where $\alpha_1 \alpha_2 \in \{1,-1\}$ is a constant. By Proposition \ref{jugcoreone}, $v_1$ and $v_j$ are CR-associated vertices in $T$, which contradicts that $T$ is a basic tournament.
\end{proof}

Let $H \in \mathcal{D}_{k} \backslash \mathcal{D}_{k-2}$ be a basic tournament with odd $k$. Then $k \geq 3$ by Proposition \ref{basicnotD1}. Now we show Theorem \ref{maintheorembasic} holds.

\begin{proof}[{\bf Proof of Theorem 4.1}]
	We complete the proof by proving ${\rm(i)} \Rightarrow {\rm(iii)}$, ${\rm(iii)}\Rightarrow {\rm(ii)}$ and ${\rm(ii)} \Rightarrow {\rm(i)}$.
	
	\textbf{Step 1}: ${\rm(ii)} \Rightarrow {\rm(i)}$.
	
	By ${\rm(ii)}$, all 1-transitive blowups of $H$ are CR tournaments, then $H$ is a strong CR tournament by Proposition \ref{strongCR1}.
	
	\textbf{Step 2}: ${\rm(iii)} \Rightarrow {\rm(ii)}$.
	
	Let $H^{*}$ be a transitive blowup of $H$. Then $H^{*} \in \xi(H) \cap (\mathcal{D}_{k} \backslash \mathcal{D}_{k-2})$ by {\rm (iii)}. Since $H$ is a basic tournament, we have $|V(H)| \geq 4$ and thus $|V(H^{*})| \geq 4$. If $H^{*}$ is a diamond, then $H^{*}$ is a CR tournament by Definition \ref{defCR}. Now we assume that $|V(H^{*})| \geq 4$ and $H^{*}$ is not a diamond, and we will show that $H^{*}$ is a CR tournament.
	
	Let $u$ be a non-CR vertex for $H^{*}$ with a non-CR dominating relation $\sigma$. Then $H^{*}(u,\sigma) \in \xi(H)$ by $H^{*} \in \xi(H)$, and $H^{*}(u,\sigma)$ can not be switching equivalent to a transitive blowup of $H$ by Proposition \ref{lem4maintheorem}, which implies that $H^{*}(u,\sigma) \notin \xi(H) \cap (\mathcal{D}_{k} \backslash \mathcal{D}_{k-2})$ by {\rm (iii)}. Therefore $H^{*}(u,\sigma) \notin \mathcal{D}_{k} \backslash \mathcal{D}_{k-2}$ by $H^{*}(u,\sigma) \in \xi(H)$, and thus $H^{*}(u,\sigma) \notin \mathcal{D}_{k}$ since $H^{*}$ is a subtournament of $H^{*}(u,\sigma)$ and $H^{*} \in \mathcal{D}_{k} \backslash \mathcal{D}_{k-2}$, which implies $H^{*}$ is a CR tournament by Definition \ref{defCR}. Thus {\rm (ii)} holds.
	
	\textbf{Step 3}: ${\rm(i)} \Rightarrow {\rm(iii)}$.
	
	Let $T$ be switching equivalent to a transitive blowup of $H$. Then $T \in \mathcal{D}_{k} \backslash \mathcal{D}_{k-2}$ by $H \in \mathcal{D}_{k} \backslash \mathcal{D}_{k-2}$, Corollaries \ref{crllblowupclass} and \ref{switchDk}, and $T \in \xi(H)$. Thus $T \in \xi(H) \cap (\mathcal{D}_{k} \backslash \mathcal{D}_{k-2})$. 
	
	Conversely, let $T$ be a tournament such that $T \in \xi(H) \cap (\mathcal{D}_{k} \backslash \mathcal{D}_{k-2})$. Now we show that $T$ is switching equivalent to a transitive blowup of $H$.
	
	Since $T \in \xi(H)$, there exists a subset $X \subseteq V(T)$ and a switch $T_1$ of $T$ such that $T_1[X]$ is isomorphic to $H$. If we can prove that $T_1$ is switching equivalent to a transitive blowup of $H$, then $T$ is switching equivalent to a transitive blowup of $H$. Hence, without loss of generality, we can assume that $T[X]$ is isomorphic to $H$.
	
	Assume that $V(T)=\{v_1,v_2,\ldots,v_n\}$ and $X=\{v_1,v_2,\ldots,v_m\}$. Since $H$ is a basic tournament, we have $n \geq m \geq 4$. If $n=m$, then it is trivial that $T$ is switching equivalent to a transitive blowup of $H$. Now we consider $n>m$.
	
	Since $H$ is a basic strong CR tournament and $T[X]$ is isomorphic to $H$, we have $T[X]\in \mathcal{D}_k \backslash \mathcal{D}_{k-2}$ and $T[X]$ is also a basic strong CR tournament by Theorem \ref{thforswiso}. For any $i>m$, $T[X\cup \{v_i\}]$ is a tournament generated by $T[X]$ and $v_i$, denoted by $T[X](v_i)$, and $T[X](v_i) \in \mathcal{D}_{k} \backslash \mathcal{D}_{k-2}$ by $T \in \mathcal{D}_{k} \backslash \mathcal{D}_{k-2}$. Then all $v_i$ ($i>m$) are CR vertices for $T[X]$ since $T[X]$ is a CR tournament, $T[X]\in \mathcal{D}_k\backslash \mathcal{D}_{k-2}$ and $T[X](v_i)\in \mathcal{D}_k\backslash \mathcal{D}_{k-2}$.
	
	Let $1\leq j\leq m$, $Y^{(j)}_{co}=\{v_k\hspace{0.15cm}|\hspace{0.15cm} \text{$v_k$ and $v_j$ are covertices in $T[X](v_k)$, $m< k\leq n$}\}$, $Y^{(j)}_{re}=\{v_k\hspace{0.15cm}|\hspace{0.15cm} \text{$v_k$ and $v_j$ are revertices in $T[X](v_k)$, $m< k\leq n$}\}$, $Y^{(j)}=Y^{(j)}_{co}\cup Y^{(j)}_{re}\cup \{v_j\}$. Then $Y^{(1)}\cup Y^{(2)}\cup \cdots \cup Y^{(m)}=V(T)$. Moreover, $Y^{(i)} \cap Y^{(j)}=\emptyset$ for $i \neq j$ by Lemma \ref{lem3maintheorem}. Hence $\{Y^{(1)}, \ldots, Y^{(m)}\}$ is a partition of $V(T)$.
	
	Let $W=Y^{(1)}_{re}\cup Y^{(1)}_{re}\cup Y^{(2)}_{re}\cup \cdots \cup Y^{(m)}_{re}$. Then $T$ is switching equivalent to $T'$ with respect to $W$ such that for $j\in \{1,2,\ldots,m\}$, $v_k$ and $v_j$ are covertices in $T'[X](v_k)$ for each $v_k\in Y^{(j)}\backslash \{v_j\}$. It is clear that $T'[X]=T[X]$, and thus $T'[X]$ is a basic strong CR tournament and $T'[X]\in \mathcal{D}_{k} \backslash \mathcal{D}_{k-2}$. Moreover, $T'\in \mathcal{D}_k\backslash \mathcal{D}_{k-2}$ by Corollary \ref{switchDk} and $T\in \mathcal{D}_{k}\backslash \mathcal{D}_{k-2}$.
	
	Take $u_1\in Y^{(i)}\backslash \{v_i\}$ and $u_2\in Y^{(j)}\backslash \{v_j\}$, where $i\in \{1,\ldots,m\}$ and $j\in \{1,\ldots,m\}\backslash \{i\}$. Now we prove that $\theta_{T'}(u_1,u_2)\cdot \theta_{T'}(v_i,v_j)=1$.
	
	Since $u_1$ and $v_i$ are covertices in $T'[X](u_1)$, $u_2$ and $v_j$ are covertices in $T'[X](u_2)$, we have $T'[X \cup \{u_1\}]$ is a 1-transitive blowup of $T'[X]$ and $T'[X \cup \{u_1\}]\in \mathcal{D}_{k} \backslash \mathcal{D}_{k-2}$ by Corollary \ref{crllblowupclass}. If $\theta_{T'}(u_1,u_2)\cdot \theta_{T'}(v_i,v_j)=-1$, then by Lemma \ref{lem1maintheorem}, we have $T'[X \cup \{u_1,u_2\}]$ is a basic tournament, which implies $u_2$ is a non-CR vertex for $T'[X \cup \{u_1\}]$. However, since $T'[X \cup \{u_1\}]$ is a CR tournament by the facts that $T'[X \cup \{u_1\}]$ is a 1-transitive blowup of $T'[X]$ and $T'[X]$ is a basic strong CR tournament ($T[X]=T'[X]$ is isomorphic to $H$), we have $T'[X \cup \{u_1,u_2\}]\notin \mathcal{D}_k$, which contradicts that $T'\in \mathcal{D}_k\backslash \mathcal{D}_{k-2}$. Thus $\theta_{T'}(u_1,u_2)\cdot \theta_{T'}(v_i,v_j)=1$, and $Y^{(i)}\rightarrow Y^{(j)}$ if $v_i\rightarrow v_j$ for $1\leq i,j\leq m$. Therefore, $T'$ is a blowup of $H$ such that $T'=T'[X](T'[Y^{(1)}],T'[Y^{(2)}],\ldots,T'[Y^{(m)}])=H(T'[Y^{(1)}],T'[Y^{(2)}],\ldots,T'[Y^{(m)}])$.
	
	If there exists $j\in \{1,2,\ldots,m\}$ such that $T'[Y^{(j)}]$ is not transitive, then there exists a 3-cycle in $T'[Y^{(j)}]$. Assume that $\{w_1,w_2,w_3\}\subset Y^{(j)}$ and $T'[\{w_1,w_2,w_3\}]$ is a 3-cycle such that $w_1\rightarrow w_2$, $w_2\rightarrow w_3$, $w_3\rightarrow w_1$. Then $T'[\{w_1\}\cup (\{v_1,v_2,\ldots,v_m\}\backslash\{v_j\})]$ is isomorphic to $H$, $T'[\{w_1,w_2,w_3\}\cup (\{v_1,v_2,\ldots,v_m\}\backslash\{v_j\})]$ is a blowup of $H$ with respect to $T'[\{v_1\}],T'[\{v_2\}],\ldots,T'[\{v_{j-1}\}],T'[\{w_1,w_2,w_3\}],T'[\{v_{j+1}\}],\ldots,T'[\{v_m\}]$. By using Lemma \ref{lem2maintheorem}, $T'[\{w_1,w_2,w_3\}\cup (\{v_1,v_2,\ldots,v_m\}\backslash\{v_j\})]$ is a basic tournament, which implies $w_3$ is a non-CR vertex for $T'[\{w_1,w_2\}\cup (\{v_1,v_2,\ldots,v_m\}\backslash\{v_j\})]$. Notice that $T'[\{w_1,w_2\}\cup (\{v_1,v_2,\ldots,v_m\}\backslash\{v_j\})]$ is a 1-transitive blowup of $H$, thus $T'[\{w_1,w_2\}\cup (\{v_1,v_2,\ldots,v_m\}\backslash\{v_j\})]$ is a CR tournament by the fact that $H$ is a strong CR tournament. Then we have $T'[\{w_1,w_2,w_3\}\cup (\{v_1,v_2,\ldots,v_m\}\backslash\{v_j\})]\notin \mathcal{D}_k$ by $T'[\{w_1,w_2\}\cup (\{v_1,v_2,\ldots,v_m\}\backslash\{v_j\})] \in \mathcal{D}_{k} \backslash \mathcal{D}_{k-2}$, which contradicts that $T'\in \mathcal{D}_k\backslash \mathcal{D}_{k-2}$.
	
	Now we have all $T'[Y^{(j)}]$ ($j=1,2,\ldots,m$) are transitive, which implies $T'$ is a transitive blowup of $H$, and thus $T$ is switching equivalent to a transitive blowup of $H$. We complete the proof.
\end{proof}

Now we show how to get the characterizations of $\mathcal{D}_3\backslash \mathcal{D}_1$ and $\mathcal{D}_5 \backslash \mathcal{D}_3$ presented in Theorem \ref{D5character} by using Theorem \ref{maintheorembasic}. By Proposition \ref{L246}, $L_4$ and $L_6$ are basic strong CR tournaments.

\begin{proposition}\label{containsL4}
	Let $T \in \mathcal{D}_3 \backslash \mathcal{D}_{1}$. Then $T \in \xi(L_4)$.
\end{proposition}

\begin{proof}
	Firstly, $T$ contains a diamond by $T\notin \mathcal{D}_{1}$ and {\rm (iii)} of Theorem \ref{D1}.
	
	By the definition of diamonds, we know there are two distinct diamonds and they are switching isomorphic. For the two distinct  diamonds, one is $L_4$, and the other is switching isomorphic to $L_4$. Therefore, $T$ contains a subtournament $H$ which is $L_4$ or is switching isomorphic to $L_4$, say, $T \in \xi(L_4)$.
\end{proof}

By Theorem \ref{Lntheorem1}, we have $L_4 \in \mathcal{D}_3 \backslash \mathcal{D}_1$. Then {\rm (ii)} of Theorem \ref{D5character} holds by the fact that $L_4$ is a basic strong CR tournament, Theorem \ref{maintheorembasic} and Proposition \ref{containsL4}.

\begin{lemma}{\rm(\!\!\cite{Dfive})}\label{T6L6}
	A $6$-tournament $T$ is switching isomorphic to $L_6$ if and only if $\det(T)=25$. 
\end{lemma}

\begin{proposition}\label{containsL6}
	Let $T \in \mathcal{D}_{5} \backslash \mathcal{D}_{3}$. Then $T \in \xi(L_{6})$.
\end{proposition}
\begin{proof}
	 By {\rm (iii)} of Theorem \ref{Dthree}, $T$ contains a $6$-subtournament $H$ such that $\det(H) > 9$. Since $T \in \mathcal{D}_{5} \backslash \mathcal{D}_{3}$ and $\det(H)$ is equal to the square of an odd integer, we have $\det(H)=25$. By Lemma \ref{T6L6}, $H$ is switching isomorphic to $L_6$. It follows that $T \in \xi(L_{6})$.
\end{proof}

By Theorem \ref{Lntheorem1}, we have $L_6 \in \mathcal{D}_5 \backslash \mathcal{D}_3$. Then {\rm (iv)} of Theorem \ref{D5character} holds by the fact that $L_6$ is a basic strong CR tournament, Theorem \ref{maintheorembasic} and Proposition \ref{containsL6}.

\begin{remark}
	When $k \geq 7$, $T \in \xi(L_{k+1})$ does not necessarily hold for a tournament $T \in \mathcal{D}_{k} \backslash \mathcal{D}_{k-2}$. For example, there is a $6$-tournament $T'$ with the skew-adjacency matrix \eqref{exampleT} such that $T' \in \mathcal{D}_{7} \backslash \mathcal{D}_{5}$,
	but $T' \notin \xi(L_8)$ since $V(T') < V(L_8)$.
\end{remark}

\section{All $L_n$ are strong CR tournaments}\label{sec-allLn}
\hspace{1.5em}To solve Question \ref{queanyDktwo}, we need to further study the properties of $L_n$ for even $n$. In this section, we show the following results.

\begin{theorem}\label{evenLnbasicCR}
	Let $n\geq 4$ be a positive even integer. Then $L_n$ is a basic strong CR tournament.
\end{theorem}

\begin{theorem}\label{allLnstrongCR}
	All $L_n$ are strong CR tournaments.
\end{theorem}

The remainder of this section is organized as follows. In Subsection \ref{subsecLn-nota}, some necessary notations and lemmas are given. In Subsection \ref{subsecLn-tools}, we present some conclusions regarding the determinant of skew-symmetric matrix, which will serve as tools in the proof of Theorem \ref{evenLnbasicCR}. In Subsection \ref{subsecLn-table}, we introduce $Z$-matrix, which is our key technique for the proof of Theorem \ref{evenLnbasicCR}. In Subsection \ref{subsecLn-pfeven}, we prove Theorem \ref{evenLnbasicCR} and Theorem \ref{allLnstrongCR}.

\subsection{Notations and lemmas}\label{subsecLn-nota}
\hspace{1.5em}To begin with, we introduce an important notation.

\begin{definition}{\rm(\!\!\cite{Dfive})}\label{defpsi}
	Let $T$ be an $n$-tournament, $X$ be a subset of $V(T)$ such that $T[X]$ is transitive and $|X|=k$. For any $u\in V(T)\backslash X$ and the ordering of $X$, $\{v_1,\ldots,v_k\}$, which satisfies $v_1\rightarrow v_2\rightarrow \cdots \rightarrow v_k$ in $T$, we define the dominating relation between $u$ and $X$ by $\psi_T(u,X)=(\alpha_1,\ldots,\alpha_t)$, where nonzero integers $\alpha_1,\ldots,\alpha_t$ and a partition $X(i,\alpha_i)(i=1,\ldots,t)$ of $X$ satisfy that $|\alpha_1|+\cdots+|\alpha_t|=k$, $\alpha_i\alpha_{i+1}<0$ for $1\leq i\leq t-1$, $X(1,\alpha_1)=\{v_1,\ldots,v_{|\alpha_1|}\}$, $X(j  ,\alpha_j)=\{v_{|\alpha_1|+\cdots+|\alpha_{j-1}|+1},\ldots,v_{|\alpha_1|+\cdots+|\alpha_j|}\}$ for $2\leq j\leq t$, and the arcs between $u$ and $X$ satisfy that $\{u\} \rightarrow X(i,\alpha_i)$ if $\alpha_i>0$, and $\{u\}\leftarrow X(i,\alpha_i)$ if $\alpha_i<0$.
\end{definition}

\begin{remark}
	We note that the notation $\psi_{T}(u,X)=(\alpha_1,\ldots,\alpha_t)$ represent the dominating relation between $u$ and the vertices in $X$ ordered by transitivity. For example, if $X=\{v_1,v_2,v_3,v_4\}$ and $v_3 \rightarrow v_4 \rightarrow v_1 \rightarrow v_2$ in $T$, then $\psi_T(u,X)=(1,-2,1)$ implies that $u \rightarrow v_3$, $u \leftarrow v_4$, $u \leftarrow v_1$, $u \rightarrow v_2$.
\end{remark}

For the sake of clarity and consistency in the subsequent discussion, throughout the remainder of Section \ref{sec-allLn}, we shall denote the vertex set of $L_n$ by $V(L_n)=\{v_1,v_2,\ldots,v_{n-1},v_n\}$, where $v_1,v_2,\ldots,v_{n-1},v_n$ satisfying that $L_n[\{v_1,v_2,\ldots,v_{n-1}\}]$ is transitive with $v_1\rightarrow v_2\rightarrow \cdots \rightarrow v_{n-1}$ and $\psi_{L_n}(v_n,\{v_1,v_2,\ldots,v_{n-1}\})=((-1)^0,(-1)^1,\ldots,(-1)^{n-2})$.

We use $L_n^{-}$ to denote the switch of $L_n$ with respect to $\{v_n\}$, consequently, we have $\psi_{L_n^{-}}(v_n,\{v_1,v_2,\ldots,v_{n-1}\})=((-1)^{1},(-1)^{2},\ldots,(-1)^{n-1})$.

For notational convenience, when no confusion arises, we abbreviate $T(u,\sigma)$ as $T(u)$ and simply say ``$u$ is a CR vertex (non-CR vertex) for $T$'' (omit the reference to $\sigma$) in the following.

For $n \geq 3$, we have the following lemma.

\begin{lemma}\label{CRveforLneven}
	Let $n \geq 3$ be a positive integer, $T \in \{L_n, L^{-}_{n}\}$, $T(u)$ be the tournament generated by $T$ and $u$ with some dominating relation and $\psi_{T(u)}(u,X)=(\alpha_1,\alpha_2,\ldots,\alpha_t)$, where $X=\{v_1,\ldots,v_{n-1}\}$. Then we have
	\item {\rm (i)} if $n$ is even, then $u$ is a CR vertex for $T$ if and only if $t\in \{1,2,n-1\}$;
	\item {\rm (ii)} if $n$ is odd, then $u$ is a CR vertex for $T$ if and only if $t\in \{2,n-1\}$, or $t=1$ with $\alpha_1 \cdot \theta_{T(u)}(u,v_n)<0$.
\end{lemma}

\begin{proof}
	Firstly, we show {\rm (i)} holds. If $u$ is a CR vertex for $T$, then there exists $v_i\in \{v_1,\ldots,v_n\}$ such that $u$ and $v_i$ are CR-associated vertices in $T(u)$. Since $T \in \{L_n, L^{-}_{n}\}$, by a direct checking, we have
	\[
	t\in \begin{cases}
		\{1,2\}, & \text{if $i \in \{1,n-1\}$};\\
		\{2\}, & \text{if $i \in \{2,3,\ldots,n-2\}$};\\
		\{n-1\}, & \text{if $i \in \{n\}$}.
	\end{cases}
	\]
	Hence $t\in \{1,2,n-1\}$.
	
	Conversely, if $t\in \{1,2,n-1\}$, we have the following cases.
	
	\textbf{Case 1}: $t=1$.
	
	If $\alpha_1>0$ (resp. $\alpha_1<0$), then by the condition that $T \in \{L_n, L^{-}_{n}\}$, we have $u$ and $v_1$ are covertices (resp. revertices) in $L_n(u)$ if $u \leftarrow v_n$ (resp. $u \rightarrow v_n$), $u$ and $v_{n-1}$ are revertices (resp. covertices) in $L_n(u)$ if $u \rightarrow v_n$ (resp. $u \leftarrow v_n$); $u$ and $v_1$ are covertices (resp. revertices) in $L^{-}_n(u)$ if $u \rightarrow v_n$ (resp. $u \leftarrow v_n$, $u$), $u$ and $v_{n-1}$ are revertices (resp. covertices) in $L^{-}_n(u)$ if $u \leftarrow v_n$ (resp. $u \rightarrow v_n$).
	
	\textbf{Case 2}: $t=2$.
	
	Let $j=|\alpha_1|$. Then by the condition that $T \in \{L_n, L^{-}_{n}\}$, we have $j<n-1$ and $\theta_{T}(v_j,v_n)\cdot \theta_{T}(v_{j+1},v_n)=-1$. Let $k_1\in \{j,j+1\}$ such that $\theta_{T}(v_{k_1},v_n)\cdot \theta_{T}(u,v_n)=1$. Then $k_2\in \{j,j+1\}\backslash \{k_1\}$ satisfies that $\theta_{T}(v_{k_2},v_n)\cdot \theta_{T}(u,v_n)=-1$. If $\alpha_1>0$, then $u$ and $v_{k_2}$ are revertices in $T(u)$; if $\alpha_1<0$, then $u$ and $v_{k_1}$ are covertices in $T(u)$.
	
	\textbf{Case 3}: $t=n-1$.
	
	In this case, for $T=L_n$, we have $u$ and $v_n$ are covertices if $\alpha_1>0$, and revertices if $\alpha_1<0$; for $T=L^{-}_n$, we have $u$ and $v_n$ are covertices if $\alpha_1<0$, and revertices if $\alpha_1>0$.
	
	Combining the above cases, if $t\in \{1,2,n-1\}$, then $u$ is a CR vertex for $T$. Thus {\rm (i)} holds.
	
	Now we show that (ii) holds. We only consider the case $T=L_n$, and the proof for $T=L^{-}_n$ is similar, so we omit it.
	
	If $u$ is a CR vertex for $L_n$, then there exists $v_i\in \{v_1,\ldots,v_n\}$ such that $u$ and $v_i$ are CR-associated vertices in $L_n(u)$. If $i\in \{1,n-1\}$, then $t=2$, or $t=1$ with $\alpha_1 \cdot \theta_{L_n(u)}(u,v_n)<0$; if $i\in \{2,3,\ldots,n-2\}$, then $t=2$; if $i=n$, then $t=n-1$. Hence $t\in \{2,n-1\}$, or $t=1$ with $\alpha_1 \cdot \theta_{L_n(u)}(u,v_n)<0$.
	
	Conversely, if $t=1$ and $\alpha_1 \cdot \theta_{L_n(u)}(u,v_n)<0$, then $u$ and $v_1$ are covertices (or $u$ and $v_{n-1}$ are revertices) in $L_n(u)$ if $\alpha_1>0$, $u$ and $v_1$ are revertices (or $u$ and $v_{n-1}$ are covertices) in $L_n(u)$ if $\alpha_1<0$; if $t=2$ or $t=n-1$, then by the similar discussions in Cases 2 and 3, we have $u$ is a CR vertex for $L_n(u)$. Thus (ii) holds.
\end{proof}

\begin{lemma}\label{Lnadd1switch}
	Let $n\geq 2$ be a positive even integer. Then $L_{n+1}$ is switching equivalent to a $1$-transitive blowup of $L_n$.
\end{lemma}

\begin{proof}
	Let $W=\{v_n\}$. Then $L_{n+1}$ is switching equivalent to $L_n(2,1,\ldots,1)=L_n(\{v_n \rightarrow v_1\},\{v_2\},\{v_3\},\ldots,\{v_{n-1}\},\{v_{n+1}\})$ with respect to $W$, where $v_1$ and $v_n$ are covertices.
\end{proof}

Now we show that $L_n$ is a basic tournament for even $n\geq 4$.

\begin{lemma}\label{Lnbasic}
	Let $n\geq 3$. Then $L_n$ is a basic tournament if $n$ is even, and $L_n$ is not a basic tournament if $n$ is odd.
\end{lemma}

\begin{proof}
	Let $n$ be even. Then $n \geq 4$. Suppose, for the sake of contradiction, that there exists $\{v_i,v_j\}\subset V(L_n)$ ($i<j$) such that $v_i$ and $v_j$ are CR-associated vertices in $L_n$. Then $(\theta_T(v_i,v_k)\cdot\theta_T(v_j,v_k))\cdot (\theta_T(v_i,v_{\ell})\cdot\theta_T(v_j,v_{\ell}))=-1$, where $k,\ell$ satisfy
	\[
	(k,\ell)=\begin{cases}
		(n-1,2) , & \text{if $i=1$, $j=n$};\\
		(1,2), & \text{if $i=n-1$, $j=n$};\\
		(i-1,i+1), & \text{if $i\in \{2,3,\ldots,n-2\}$, $j=n$};\\
		(n-1,n), & \text{if $1\leq i <j \leq n-2$ and $j=i+1$};\\
		(1,n), & \text{if $i=n-2$ and $j=n-1$};\\
		(2,n), & \text{if $i=1$, $j=n-1$};\\
		(n-1,2), & \text{if $i=1$ and $3 \leq j \leq n-2$};\\
				(1,i+1), & \text{if $2\leq i <j \leq n-1$ and $j \neq i+1$}.
	\end{cases}
	\] By Corollary \ref{jugcoretwo}, there is a contradiction. Thus, there exist no such $v_i$ and $v_j$. It follows that $L_n$ is a basic tournament for even $n\geq 4$.

	Let $n$ be odd. Then by Lemma \ref{Lnadd1switch} and Theorem \ref{thforswiso}, $L_n$ is not a basic tournament.
\end{proof}

\begin{lemma}{\rm(\!\!\cite{Dfive})}\label{ninedet}
	Let $ T $ be an $n$-tournament $(n\geq 2)$ with vertices $v_1,\ldots,v_n$, $ H_1,\ldots,H_n $ be tournaments. If there exists $ H_i $ such that $ H_i $ is not transitive for some $i$ $(1\leq i\leq n)$, then there exists a subtournament $ T_{sub} $ of $ T(H_1,\ldots,H_n) $ such that $\det(T_{sub})=9\cdot\det(T)$. Especially, if $H_i$ is a $3$-cycle and $|V(H_j)|=1$ for $j \neq i$, then $\det(T(H_1,\ldots,H_n))=9\cdot\det(T)$.
\end{lemma}	

\begin{lemma}\label{Lnadd1CR}
	Let $n$ be a positive even integer. If $L_n$ is a CR tournament, then $L_{n+1}$ is a CR tournament. 
\end{lemma}

\begin{proof}
	When $n\in \{2,4,6\}$, $L_n$ is a strong CR tournament by Proposition \ref{L246}, and thus $L_{n+1}$ is a CR tournament since $L_{n+1}$ is switching equivalent to a $1$-transitive blowup of $L_n$ by Lemma \ref{Lnadd1switch}. Next, we consider $n\geq 8$.
	
	Let $V(L_{n+1})=\{v_1,v_2,\ldots,v_n,v_{n+1}\}$ and $X=\{v_1,v_2,\ldots,v_n\}$ such that $L_{n+1}[X]$ is transitive with $v_1\rightarrow v_2\rightarrow \cdots \rightarrow v_n$ and $\psi_{L_{n+1}}(v_{n+1},X)=(1,-1,\ldots,(-1)^{n-2},(-1)^{n-1})$. By Lemma \ref{Lnadd1switch}, $L_{n+1}$ is switching equivalent to a $1$-transitive blowup of $L_n$. Then by Theorem \ref{Lntheorem1} and Corollary \ref{crllblowupclass}, we have $L_n \in \mathcal{D}_{n-1} \backslash \mathcal{D}_{n-3}$ and $L_{n+1}\in \mathcal{D}_{n-1}\backslash \mathcal{D}_{n-3}$.
	
	Let $u$ be a non-CR vertex for $L_{n+1}$, $L_{n+1}(u)$ be the tournament generated by $L_{n+1}$ and $u$, and $\psi_{L_{n+1}(u)}(u,X)=(\alpha_1,\alpha_2,\ldots,\alpha_t)$. Then by Lemma \ref{CRveforLneven}, we have $t \in \{3,4,\ldots,n-1\}$, or $t=1$ and $\alpha_1 \cdot \theta_{L_{n+1}(u)}(u,v_{n+1})>0$. Now we show $L_{n+1}(u)\notin \mathcal{D}_{n-1}$.
	
	\textbf{Case 1}: $u$ is a non-CR vertex for $L_{n+1}[V(L_{n+1})\backslash \{v_n\}]$ or $L_{n+1}[V(L_{n+1})\backslash \{v_1\}]$.
	
	It is easy to see that $L_{n+1}[V(L_{n+1})\backslash\{v_n\}]$ is $L_n$ and $L_{n+1}[V(L_{n+1})\backslash\{v_1\}]$ is $L_n^{-}$. Since $L_n$ and $L^{-}_{n}$ are CR tournaments ($L_n^{-}$ is a switch of $L_n$), we have $L_{n+1}(u)[(V(L_{n+1})\backslash \{v_n\})\cup \{u\}]\notin \mathcal{D}_{n-1}$ or $L_{n+1}(u)[(V(L_{n+1})\backslash \{v_1\}) \cup \{u\}]\notin \mathcal{D}_{n-1}$ by the fact $L_n,L^{-}_n \in \mathcal{D}_{n-1} \backslash \mathcal{D}_{n-3}$ and Definition \ref{defCR}, which implies $L_{n+1}(u)\notin \mathcal{D}_{n-1}$.
	
	\textbf{Case 2}: $u$ is a CR vertex for $L_{n+1}[V(L_{n+1})\backslash \{v_n\}]$ and $L_{n+1}[V(L_{n+1})\backslash \{v_1\}]$.
	
	\textbf{Subcase 2.1}: $t=1$ and $\alpha_1 \cdot \theta_{L_{n+1}(u)}(u,v_{n+1})>0$.
	
	Let $W=\{v_{n+1}\}$ if $\alpha_1>0$ and $W=\emptyset$ if  $\alpha_1<0$. Then $L_{n+1}(u)$ is switching equivalent to $L_{n+2}$ with respect to $W$. Thus $L_{n+1}(u)\notin \mathcal{D}_{n-1}$ in this subcase by $L_{n+2} \in \mathcal{D}_{n+1}\backslash \mathcal{D}_{n-1}$ (Theorem \ref{Lntheorem1}) and Corollary \ref{switchDk}.
	
	\textbf{Subcase 2.2}: $t=3$.
	
	Firstly, we show $|\alpha_1|=|\alpha_3|=1$. Otherwise, if there exists $i\in \{1,3\}$ such that $|\alpha_i|>1$, we take $j=\begin{cases}
		1, & \text{if $i=1$}; \\
		n, & \text{if $i=3$},
	\end{cases}$ then $u$ is a non-CR vertex for $L_{n+1}[V(L_{n+1})\backslash \{v_j\}]$ by Lemma \ref{CRveforLneven}, which contradicts the given condition that $u$ is a CR vertex for $L_{n+1}[V(L_{n+1})\backslash \{v_n\}]$ and $L_{n+1}[V(L_{n+1})\backslash \{v_1\}]$. Therefore, $|\alpha_1|=1$, $|\alpha_2|=n-2$, $|\alpha_3|=1$.
	
	Let $W=\{v_n\}$ if $\alpha_1<0$ and $W=\{u,v_n\}$ if $\alpha_1>0$. Then $L_{n+1}(u)$ is switching equivalent to $L_{n+1}'(u)$ with respect to $W$ such that $L_{n+1}'(u)[V(L_{n+1})]$ is a $1$-transitive blowup of $L_n$ (where $v_n$ and $v_1$ are covertices), $L_{n+1}'(u)[\{v_n,v_1,\ldots,v_{n-1}\}]$ is transitive with $v_n \rightarrow v_1 \rightarrow \cdots \rightarrow v_{n-1}$, and $\psi_{L_{n+1}'(u)}(u,\{v_n,v_1,\ldots,v_{n-1}\})=(1,-1,n-2)$.
	
	If $v_{n+1}\rightarrow u$ in $L_{n+1}'(u)$, then $L_{n+1}'(u)$ is a blowup of $L_n$, that is, $L_{n+1}'(u)=L_n(T_1,\ldots,T_n)$ with respect to $T_1=L_{n+1}'(u)[\{v_n,v_1,u\}]$, $T_k=L_{n+1}'(u)[\{v_k\}]$ for $2\leq k \leq n-1$, and $T_n=L_{n+1}'(u)[\{v_{n+1}\}]$. Note that $L_{n+1}'(u)[\{v_n,v_1,u\}]$ is a $3$-cycle, we have $L_{n+1}'(u)[\{v_n,v_1,u\}]$ is not transitive. Then by Lemmas \ref{Minors} and \ref{ninedet}, we have $\det(L_{n+1}(u))=\det(L_{n+1}'(u))=9\cdot\det(L_n)=9(n-1)^2>(n-1)^2$, which implies $L_{n+1}'(u)\notin \mathcal{D}_{n-1}$ and $L_{n+1}(u)\notin \mathcal{D}_{n-1}$.
	
	If $v_{n+1}\leftarrow u$ in $L_{n+1}'(u)$, let $Z=(V(L_{n+1})\backslash \{v_n,v_2\})\cup \{u\}$. It is clear that $L_{n+1}'(u)[Z]$ is $L_n$ with $v_1 \rightarrow u \rightarrow v_3 \rightarrow \cdots \rightarrow v_{n-1}$. Now $\psi_{L_{n+1}'(u)}(v_n,\{v_1,u,v_3,\ldots,v_{n-1}\})=(1,-1,n-3)$. Therefore $v_n$ is a non-CR vertex for $L_{n+1}'(u)[Z]$ by Lemma \ref{CRveforLneven}, which implies $L_{n+1}'(u)[Z \cup \{v_n\}]\notin \mathcal{D}_{n-1}$ by the given condition that $L_n$ is a CR tournament, it follows that $L_{n+1}'(u)\notin \mathcal{D}_{n-1}$ and $L_{n+1}(u)\notin \mathcal{D}_{n-1}$.
	
	\textbf{Subcase 2.3}: $4\leq t\leq n-1$.
	
	Note that $L_{n+1}[V(L_{n+1})\backslash\{v_n\}]$ is $L_n$ and $L_{n+1}[V(L_{n+1})\backslash\{v_1\}]$ is $L_n^{-}$. When $4 \leq t \leq n-2$, it is easy to see that there exists $i\in \{1,n\}$ such that $\psi_{L_{n+1}(u)}(u,X\backslash \{v_i\})=(\beta_1,\beta_2,\ldots,\beta_s)$, where $3 \leq s \leq n-2$. When $t=n-1$, we have $|\alpha_1|=1$ or $|\alpha_{n}|=1$. Let $i=\begin{cases}
		1, & \text{if $|\alpha_1|=1$};\\
		n, & \text{if $|\alpha_{n-1}|=1$}.
	\end{cases}$ Then $s=n-2$. Thus $u$ is a non-CR vertex for $L_{n+1}[V(L_{n+1})\backslash \{v_i\}]$ by Lemma \ref{CRveforLneven}, which contradicts the given condition that $u$ is a CR vertex for $L_{n+1}[V(L_{n+1})\backslash \{v_n\}]$ and $L_{n+1}[V(L_{n+1})\backslash \{v_1\}]$.
	
	Combining the above cases, we have $L_{n+1}(u)\notin \mathcal{D}_{n-1}$, it follows that $L_{n+1}$ is a CR tournament. We complete the proof.
\end{proof}

\begin{lemma}\label{Lnstrong}
	Let $n$ be a positive even integer. If $L_n$ is a CR tournament, then $L_n$ is a strong CR tournament.
\end{lemma}

\begin{proof}
	By Proposition \ref{L246}, $L_2$, $L_4$ and $L_6$ are strong CR tournaments. Next, we consider $n\geq 8$.
	
	If $R$ is a $1$-transitive blowup of $L_n$, then there exist positive integers $a_1,a_2,\ldots,a_n$ corresponding to $v_1,v_2,\ldots,v_n$ such that $R=L_n(a_1,a_2,\ldots,a_n)$, where $|a_i|=2$ for some $i$, and $|a_j|=1$ for $j\in \{1,2,\ldots,n\}\backslash \{i\}$. Let $X_k$ ($k=1,2,\ldots,n$) denote the vertex subset of $V(R)$ corresponding to $a_k$. Now we show $R$ is a CR tournament. Clearly, $R\in \mathcal{D}_{n-1} \backslash \mathcal{D}_{n-3}$ by Theorem \ref{Lntheorem1} and Corollary \ref{crllblowupclass}.
	
	\textbf{Case 1}: $i=1$.
	
	In this case, $R=L_n(2,1,1,\ldots,1)$. Let $X_1=\{w_1,w_2\}$ such that $w_1 \rightarrow w_2$ in $R$. Then $R$ is switching equivalent to $L_{n+1}$ with respect to $\{w_1\}$, and thus $R$ is a CR tournament by Lemma \ref{Lnadd1CR} and Theorem \ref{thforswiso}.
	
	\textbf{Case 2}: $i\in \{2,3,\ldots,n-1\}$.
	
	Let $W=X_1\cup X_2\cup \cdots \cup X_{i-1}$ if $i$ is odd, $W=X_1\cup X_2\cup \cdots \cup X_{i-1}\cup X_{n}$ if $i$ is even. Then $R$ is switching equivalent to $R'$ with respect to $W$, where $R'=L_n(a_i,a_{i+1},\ldots,a_{n-1},a_1,\ldots,a_{i-1},a_n)=L_n(2,1,1,\ldots,1,1)$. By Case 1, $R'$ is a CR tournament, and thus $R$ is a CR tournament by Theorem \ref{thforswiso}.
	
	\textbf{Case 3}: $i=n$.
	
	Let $V(R)=\{w_1,w_2,\ldots,w_{n+1}\}$ such that $X_k=\{w_k\}$ for $k\in \{1,2,\ldots,n-1\}$, $X_n=\{w_n,w_{n+1}\}$ and $w_n\rightarrow w_{n+1}$.
	
	Suppose that $u$ is a non-CR vertex for $R$ with a dominating relation $\sigma$, where $\psi_{R(u)}(u,\\\{w_1,w_2,\ldots,w_{n-1}\})=(\alpha_1,\alpha_2,\ldots,\alpha_t)$ and $R(u)=R(u,\sigma)$. If $\alpha_1<0$, then $R(u)$ is switching equivalent to $R'(u)$ with respect to $\{u\}$ such that $\psi_{R'(u)}(u,\{w_1,w_2,\ldots,w_{n-1}\})=(-\alpha_1,-\alpha_2,\ldots,-\alpha_t)$, where $R'(u)=R(u,\sigma')$ and $\sigma'$ is another dominating relation between $u$ and $V(R)$. By Corollary \ref{switchDk}, we have $R(u)\notin \mathcal{D}_{n-1}$ if and only if $R'(u)\notin \mathcal{D}_{n-1}$. Without loss of generality, we assume that $\alpha_1>0$. In the following, we prove that $R$ is a CR tournament by showing that $R(u)\notin \mathcal{D}_{n-1}$.
	
	\textbf{Subcase 3.1}: $t=1$.
	
	In this subcase, $\psi_{R(u)}(u,\{w_1,w_2,\ldots,w_{n-1}\})=(n-1)$. It is easy to check that $u$ is a non-CR vertex for $R$ if and only if $\theta_{R(u)}(u,w_n)\cdot \theta_{R(u)}(u,w_{n+1})=-1$. Then there exist $m_1 \in \{n,n+1\}$ and $m_2 \in \{n,n+1\} \backslash \{m_1\}$ such that $\theta_{R(u)}(u,w_{m_1})=1$ and $\theta_{R(u)}(u,w_{m_2})=-1$.
	
	Let $Z=(\{w_1,w_2,\ldots,w_{n-1}\}\backslash\{w_1\})\cup \{u\}$. Then $R(u)[Z\cup \{w_{m_2}\}]$ is $L_n$ and $R(u)[Z]$ is transitive with $u\rightarrow w_2\rightarrow w_3\rightarrow \cdots \rightarrow w_{n-1}$. By Theorem \ref{Lntheorem1}, we have $R(u)[Z\cup \{w_{m_2}\}] \in \mathcal{D}_{n-1}\backslash \mathcal{D}_{n-3}$. Note that $\psi_{R(u)}(w_{m_1},Z)=(\beta_1,\beta_2,\ldots,\beta_{n-2})=(-2,(-1)^{2},\ldots,(-1)^{n-2})$. Then $w_{m_1}$ is a non-CR vertex for $R(u)[Z\cup \{w_{m_2}\}]$ by Lemma \ref{CRveforLneven}, and we have $R(u)[Z\cup \{w_{m_1},w_{m_2}\}]\notin \mathcal{D}_{n-1}$ by the given condition that $R(u)[Z\cup \{w_{m_2}\}]$ ($\cong L_n$) is a CR tournament, it follows that $R(u)\notin \mathcal{D}_{n-1}$.
	
	\textbf{Subcase 3.2}: $t=2$.
	
	Let $W=\{w_1,w_2,\ldots,w_{\alpha_1}\}\cup \{u\}$ if $\alpha_1$ is even, and $W=\{w_1,w_2,\ldots,w_{\alpha_1}\}\cup \{u\} \cup X_n$ if $\alpha_1$ is odd. Then $R(u)$ is switching equivalent to $R'(u)$ with respect to $W$ such that $R'(u)[V(R)]$ is isomorphic to $R$, $R'(u)[\{w_{\alpha_1+1},\ldots,w_{n-1},w_1,\ldots,w_{\alpha_1}\}]$ is transitive with $w_{\alpha_1+1}\rightarrow w_{\alpha_1+2}\rightarrow \cdots \rightarrow w_{n-1}\rightarrow w_{1}\rightarrow \cdots \rightarrow w_{\alpha_1}$ and $\psi_{R'(u)}(u,\{w_{\alpha_1+1},\ldots,w_{n-1},w_1,\ldots,\\w_{\alpha_1}\})=(n-1)$. Therefore, $R(u)$ is switching isomorphic to the tournament discussed in Subcase 3.1, and we have $R(u) \notin \mathcal{D}_{n-1}$ by using Corollary \ref{switchDk}.
	
	\textbf{Subcase 3.3}: $t\in \{3,\ldots,n-2\}$.
	
	Let $Z=\{w_1,w_2,w_3,\ldots,w_{n-1},w_n\}$. Then $R[Z]$ is $L_n$, and $u$ is a non-CR vertex for $R[Z]$ by Lemma \ref{CRveforLneven}. Therefore, $R[Z \cup \{u\}]\notin \mathcal{D}_{n-1}$ by $L_n \in \mathcal{D}_{n-1} \backslash \mathcal{D}_{n-3}$ and the given condition that $L_n$ is a CR tournament, it follows that $R(u) \notin \mathcal{D}_{n-1}$.
	
	\textbf{Subcase 3.4}: $t=n-1$.
	
	Since $\alpha_1>0$, we have $R(u)$ is a blowup of $L_n$ with respect to $H_1,H_2,\ldots,H_n$, where $H_i=R(u)[\{w_i\}]$ for $i\in \{1,\ldots,n-1\}$ and $H_n=R(u)[\{w_n,w_{n+1},u\}]$. Now $u$ is a non-CR vertex for $L_n$ if and only if $R(u)[\{w_n,w_{n+1},u\}]$ is a $3$-cycle, then by Lemma \ref{ninedet}, we have $\det(R(u))=9\cdot \det(L_n)=9\cdot(n-1)^2$, which implies $R(u) \notin \mathcal{D}_{n-1}$.
	
	Combining the above arguments, $R$ is a CR tournament. Therefore, when even $n\geq 8$, $L_n$ is a strong CR tournament if $L_n$ is a CR tournament. We complete the proof.
\end{proof}

\subsection{Tools}\label{subsecLn-tools}
\begin{lemma}{\rm(\!\!\text{ Schur complement}\cite{MatrixAly})}\label{schur}
	Let $M_{11}$ and $M_{22}$ be square matrices such that $M_{22}$ is invertible, and $M$ be the block matrix as follows:
	
	\[M = \left[ \begin{array}{c|c} 
		M_{11} & M_{12}\\ \hline
		M_{21} & M_{22} 
	\end{array} \right].\]	
	Then $\det(M)=\det(M/M_{22}) \cdot \det(M_{22})$, where $M/M_{22}=M_{11}-M_{12}M_{22}^{-1}M_{21}$ is the schur complement of $M_{22}$.
\end{lemma}

\begin{lemma}\label{skewproperty}
	Let $S$ be a skew-symmetric matrix of order $n$, $x$ and $y$ be vectors of order $n$. Then $x^{\mathsf{T}}Sy=-y^{\mathsf{T}}Sx$. In particular, if $x=y$, then $x^{\mathsf{T}}Sy=0$.
\end{lemma}

\begin{proof}
	Since $S$ is a skew-symmetric matrix, we have $x^{\mathsf{T}}Sy=(x^{\mathsf{T}}Sy)^{\mathsf{T}}=y^{\mathsf{T}}S^{\mathsf{T}}x=y^{\mathsf{T}}(-S)x=-y^{\mathsf{T}}Sx$.
	
	If $x=y$, then $x^{\mathsf{T}}Sx=-x^{\mathsf{T}}Sx$, thus we have $2x^{\mathsf{T}}Sx=0$, which implies $x^{\mathsf{T}}Sx=0$.
\end{proof}

\begin{proposition}{\rm(\!\!\cite{DETTRANSI,SEDTOUR})}\label{dettranone}
	Let $\mathbb{T}$ be a transitive tournament of order $n$. Then $\det(\mathbb{T})=1$ if $n$ is even, $\det(\mathbb{T})=0$ if $n$ is odd.
\end{proposition}

\begin{proposition}{\rm(\!\!\cite{SEDTOUR})}\label{transpose}
	Let $n$ be a positive even integer, $\mathbb{T}$ be a transitive tournament of order $n$, $V(\mathbb{T})=\{v_1,v_2,\ldots,v_n\}$ such that $v_i\rightarrow v_j$ if $i<j$, $S_{\mathbb{T}}$ be the skew-adjacency matrix of $\mathbb{T}$ with respect to the vertex ordering $v_1,v_2,\ldots,v_n$. Then
	
	\[\scriptsize
	S_{\mathbb{T}}^{-1} = 
	\left[
	\begin{array}{rrrrrrrr}
		0 & -1 & 1 & -1 & 1 & \cdots & 1 & -1 \\
		1 & 0 & -1 & 1 & -1 & 1 & & 1 \\
		-1 & 1 & 0 & -1 & 1 &\ddots & \ddots & \vdots \\
		1 & -1 & 1 & 0 & -1 & \ddots & \ddots & 1 \\
		-1 & 1 & -1 & 1 & 0 & \ddots & \ddots & -1 \\
		\vdots & \ddots & \ddots & \ddots & \ddots & \ddots & -1 & 1 \\
		-1 &  & \ddots & \ddots & \ddots & 1 & 0 & -1 \\
		1 & -1 & \cdots & -1 & 1 & -1 & 1 & 0
	\end{array}
	\right].
	\]
\end{proposition}

Proposition \ref{toolforS} appears in the proof of \cite[Theorem 3.2]{SEDTOUR}. For the needs of the remainder of this section, we state it here as a conclusion and provide its proof.

\begin{proposition}{\rm(\!\!\cite{SEDTOUR})}\label{toolforS}
	Let $p$ be a positive even integer, $x,y$ be vectors of order $p$, $a$ be a real number, $\mathbb{T}$ be a transitive tournament of order $p$, and $S_{\mathbb{T}}=[s_{ij}]$ be the skew-adjacency matrix of $\mathbb{T}$ such that \[s_{ij}=\begin{cases}
		1, & i<j\\
		0, & i=j\\
		-1, & i>j
	\end{cases},
	\]
	$S$ be a  skew-symmetric matrix of order $p+2$ such that
	\[
	S= \left[ \begin{array}{cc|c}
		0 & a & x^{\mathsf{T}}\\ 
		-a & 0 & y^{\mathsf{T}}\\ \hline
		-x & -y & S_{\mathbb{T}} \end{array} \right].
	\]
	Then we have
	\item{\rm(i)} $\det(S)=(a+x^{\mathsf{T}}S_{\mathbb{T}}^{-1}y)^2$.
	\item{\rm(ii)} If $x^{\mathsf{T}}=
	\left[ \begin{array}{ccccc}
		1 & -1 & 1 & \cdots & -1 \end{array} \right]$ and $y^{\mathsf{T}}=\left[ \begin{array}{ccccc}
		\beta_1 & \beta_2 & \beta_3 & \cdots & \beta_{p} \end{array} \right]$, then
	\[
	\det(S)=(a+\sum_{i=1}^{p}(-1)^i \cdot (p+1-2i)\cdot\beta_i)^2.
	\]
\end{proposition}

\begin{proof}
	Since $S_{\mathbb{T}}$ is a skew-adjacency matrix of $\mathbb{T}$, then $S_{\mathbb{T}}$ is a skew-symmetric matrix, consequently, $S_{\mathbb{T}}^{-1}$ is a skew-symmetric matrix. By Lemma \ref{skewproperty}, we have
	\begin{align}
		\text{$y^{\mathsf{T}}S_{\mathbb{T}}^{-1}x=-x^{\mathsf{T}}S_{\mathbb{T}}^{-1}y$, \, $x^{\mathsf{T}}S_{\mathbb{T}}^{-1}x=y^{\mathsf{T}}S_{\mathbb{T}}^{-1}y=0$.} \label{ak1}
	\end{align} 
	
	By Proposition \ref{dettranone}, we have
	\begin{align}
		\det(S_{\mathbb{T}})=1.\label{ak3}
	\end{align} 
	
	By using Lemma \ref{schur}, (\ref{ak1}) and (\ref{ak3}), we have
	\begin{align}
		\det(S)&=\det(S_{\mathbb{T}})\cdot \det(S/S_{\mathbb{T}})\notag\\
		&=\det(S_{\mathbb{T}})\cdot \det \left( \left[ \begin{array}{cc}
			0 &a \\ -a & 0 \end{array}
		\right] - \left[ \begin{array}{c} x^{\mathsf{T}} \\ y^{\mathsf{T}} 
		\end{array} \right] S_{\mathbb{T}}^{-1} \left[ \begin{array}{cc} 
			-x & -y \end{array} \right] \right) \notag\\
		&=(a+x^{\mathsf{T}} S_{\mathbb{T}}^{-1}y)^2.\notag \label{ak4}
	\end{align}
	
	Therefore, (i) holds.
	
	By Proposition \ref{transpose}, it is easy to check that  \[ 
	x^{\mathsf{T}}S_{\mathbb{T}}^{-1}=
	\left[ \begin{array}{rrrrrrrrr}
		x_1 & x_2 & x_3 & \cdots & x_{\frac{p}{2}} & x_{\frac{p}{2}+1}  & \cdots & x_{p-1} & x_p\end{array} \right],
	\]
	where $x_i=\sum_{i=1}^{p}(-1)^i \cdot (p+1-2i)$. Then we have
	\begin{align}
		x^{\mathsf{T}}S_{\mathbb{T}}^{-1}y=\sum_{i=1}^{p}(-1)^i \cdot (p+1-2i)\cdot\beta_i \label{ak5}
	\end{align}
	
	Combining (i) and (\ref{ak5}), we have
	\begin{align}
		\det(S)&=(a+x^{\mathsf{T}}S_{\mathbb{T}}^{-1}y)^2 \notag \\
		&=(a+\sum_{i=1}^{p}(-1)^i \cdot (p+1-2i)\cdot\beta_i)^2.
	\end{align}
	
	Therefore, (ii) holds.
\end{proof}

\subsection{$Z$-matrix and its properties}\label{subsecLn-table}

\hspace{1.5em}To solve a technical part in the proof of Theorem \ref{evenLnbasicCR}, we define the $Z$-matrix and investigate its properties in this subsection. The core idea is to transform the algebraic problems involved in the proof of Theorem \ref{evenLnbasicCR} into the numerical variation problems on the $Z$-matrix, and solve them by using the properties of $Z$-matrix.

A $Z$-matrix, with respect to a positive odd integer $m\,(\geq 3)$ and a $\{1,-1\}$-sequence $r=(r_1,r_2,\ldots,r_{m})$, is an $m \times (m-1)$ matrix in which every element is an integer, denoted by $Z(m,r)$ and defined as follows.

\begin{definition}\label{defZmatrix}
	Let $m \geq 3$ be a positive odd integer and $r=(r_1,r_2,\ldots,r_{m})$ be a $\{1,-1\}$-sequence. Define the matrix $Z(m,r)=[z_{ij}]_{m \times (m-1)}$ by the following:
	\[
	z_{ij}=\begin{cases}
		(-1)^{i+j}\cdot (m-2j)\cdot r_{i+j}, & \text{if \,$i+j \leq m$};\\
		(-1)^{i+j-m}\cdot (m-2j)\cdot r_{i+j-m}, & \text{if \,$i+j> m$}.
	\end{cases}
	\]
\end{definition}

The $\ell$-antidiagonal vector of $Z(m,r)$ is a vector of order $m$, defined as follows.

\begin{definition}\label{Zmatrixl-diag}
	Let $Z(m,r)=[z_{ij}]$ be a $Z$-matrix, where $r=(r_1,r_2,\ldots,r_m)$. Define the vector $\Gamma_{\ell} \,(\ell \in \{1,2,\ldots,m\})$ by the following:
	\[
	\Gamma_{\ell}=(\gamma^{(\ell)}_1,\ldots,\gamma^{(\ell)}_{m})^{\mathsf{T}}, \quad \text{where $1 \leq \ell \leq m$ and }\gamma^{(\ell)}_i=\begin{cases}
		z_{i(\ell-i)}, & \text{if \,$i< \ell$};\\
		0, & \text{if \,$i=\ell$};\\
		z_{i(m+\ell-i)}, & \text{if \,$i> \ell$}.
	\end{cases}
	\]
	We call $\Gamma_{\ell}$ the $\ell$-antidiagonal vector of $Z(m,r)$.
\end{definition}

Here we provide an example of a $Z$-matrix and show its $\ell$-antidiagonal vectors. Let $m=9$ and the $\{1,-1\}$-sequence $r=(1,1,1,-1,-1,-1,1,-1,-1)$. Then
\begin{align*}
	\scriptsize Z(m,r) =  \left[ \begin{array}{rrrrrrrr}
		7 & -5 & -3 & 1 & 1 & 3 & 5 & -7 \\
		-7 & -5 & 3 & -1 & 1 & 3 & -5 & 7\\
		-7 & 5 & -3 & -1 & 1 & -3 & 5 & -7\\
		7 & -5 & -3 & -1 & -1 & 3 & -5 & 7 \\
		-7 & -5 & -3 & 1 & 1 & -3 & 5 & 7\\
		-7 & -5 & 3 & -1 & -1 & 3 & 5 & -7\\
		-7 & 5 & -3 & 1 & 1 & 3 & -5 & 7\\
		7 & -5 & 3 & -1 & 1 & -3 & 5 & 7\\
		-7 & 5 & -3 & -1 & -1 & 3 & 5 & 7
	\end{array}\right],
\end{align*}
and $\Gamma_{1},\Gamma_{2},\ldots,\Gamma_{n}$ are the first, second, …, 
$n$-th column vectors of the matrix below:
\begin{align*}
\scriptsize	(\Gamma_{1},\Gamma_{2},\ldots,\Gamma_{n})=\left[ \begin{array}{rrrrrrrrr}
		0 & 7 & -5 & -3 & 1 & 1 & 3 & 5 & -7\\
		7 & 0 & -7 & -5 & 3 & -1 & 1 & 3& -5\\
		5 & -7 & 0 & -7 & 5 & -3 & -1 & 1& -3\\
		3 & -5 & 7 & 0 & 7 & -5 & -3 & -1& -1 \\
		1 & -3 & 5 & 7 & 0 & -7 & -5 & -3& 1\\
		-1 & -1 & 3 & 5 & -7 & 0 & -7 & -5& 3\\
		-3 & 1 & 1 & 3 & -5 & 7 & 0 & -7& 5\\
		-5 & 3 & -1 & 1 & -3 & 5 & 7 & 0& 7\\
		-7 & 5 & -3 & -1 & -1 & 3 & 5 & 7& 0
	\end{array}\right]
\end{align*}

\begin{proposition}\label{p1Zmatrix}
	Let $r=(r_1,\ldots,r_m)$, $Z(m,r)=[z_{ij}]$ be a $Z$-matrix, $\Gamma_{\ell}=(\gamma^{(\ell)}_1,\ldots,\gamma^{(\ell)}_{m})^{\mathsf{T}}$ $(\ell \in \{1,2,\ldots,m\})$ be the $\ell$-antidiagonal vectors of $Z(m,r)$ and the $m \times m$ matrix $\Gamma=(\Gamma_1,\ldots,\Gamma_{m})$.
	\item{\rm(i)} Let $J_{m-1}$ and $J_{m}$ be the all-ones vectors. Then\[Z(m,r)\cdot J_{m-1}= \Gamma\cdot J_{m}.\]
	\item{\rm(ii)} For $i \in \{1,2,\ldots,m-1\} \backslash \{\ell -1, \ell\}$, we have
		\begin{align*}
		\gamma^{(\ell)}_{i+1}-\gamma^{(\ell)}_{i}=2\cdot (-1)^{\ell}\cdot r_{\ell}.
	\end{align*}
\end{proposition}

\begin{proof}
	Let $b_i$ be the $i$-th element of $Z(m,r)\cdot J_{m-1}$ and $c_i$ be the $i$-th element of $\Gamma\cdot J_{m}$. Now we show $b_i=c_i$ by the following computation.
	\begin{align}
		c_i&=\sum^{m}_{\ell=1}\gamma^{(\ell)}_{i}=\sum_{\ell<i}z_{i(m+\ell-i)}+0+\sum_{\ell>i}z_{i(\ell-i)} \notag \\
		&=z_{i(m+1-i)}+\cdots+z_{i(m-1)}+z_{i1}+\cdots+z_{i(m-i)}=\sum_{1\leq \ell \leq m-1}z_{i\ell}=b_i. \label{jm}
	\end{align}
	
	This completes the proof of {\rm(i)}.
	
	When $1 \leq i \leq \ell-2$, we have
	\begin{align*}
		\gamma^{(\ell)}_{i+1}-\gamma^{(\ell)}_{i}&=z_{(i+1)(\ell-i-1)}-z_{i(\ell-i)}\\
		&=(-1)^{\ell}\cdot (m-2\ell+2i+2)\cdot r_{\ell}-(-1)^{\ell}\cdot (m-2\ell+2i)\cdot r_{\ell}\\
		&=2\cdot (-1)^{\ell}\cdot r_{\ell}.
	\end{align*}
	
	When $\ell+1 \leq i \leq m-1$, we have
	\begin{align*}
		\gamma^{(\ell)}_{i+1}-\gamma^{(\ell)}_{i}&=	z_{(i+1)(m+\ell-i-1)}-z_{i(m+\ell-i)}\\
		&=(-1)^{\ell}\cdot (-m-2\ell+2i+2)\cdot r_{\ell}-(-1)^{\ell}\cdot (-m-2\ell+2i)\cdot r_{\ell}\\
		&=2\cdot (-1)^{\ell}\cdot r_{\ell}.
	\end{align*}
	
	This completes the proof of {\rm(ii)}.
\end{proof}

For convenience, we denote $\Delta(\Gamma_{\ell})=2\cdot (-1)^{\ell}\cdot r_{\ell}$ for $1 \leq \ell \leq m$. Clearly, $\Delta(\Gamma_{\ell}) \in \{2,-2\}$, and when $r=(r_1,r_2,\ldots,r_m)$ is given, the value of $\Delta(\Gamma_{\ell})$ depends only on $\ell$.

Similar to the notation in Definition \ref{defpsi}, we can uniquely represent the $\{1,-1\}$-sequence $(r_1, r_2, \ldots, r_n)$ by $(\alpha_1, \ldots, \alpha_t)$, where the representation is a bijection and $|\alpha_1| + |\alpha_2| + \cdots + |\alpha_t| = n$. For example, if $\sigma=(1,1,-1,-1,-1,1,1)$, then we can use $(2,-3,2)$ to represent $\sigma$ (vice versa). For convenience, in the following (including the next subsection), we will use $(\alpha_1,\ldots,\alpha_t)$ to represent a $\{1,-1\}$-sequence $(r_1,r_2,\ldots,r_m)$.

\begin{theorem}\label{thm1Zmatrix}
	Let $r=(r_1,r_2,\ldots,r_m)=(\alpha_1,\alpha_2,\ldots,\alpha_t)$, $Z(m,r)=[z_{ij}]$ be a $Z$-matrix, $\Gamma_{\ell}=(\gamma^{(\ell)}_1,\ldots,\gamma^{(\ell)}_{m})^{\mathsf{T}}$ ($1 \leq \ell \leq m$) be the $\ell$-antidiagonal vectors of $Z(m,r)$, $Z(m,r)\cdot J_{m-1}=(b_1,b_2,\ldots,b_m)$, $\mathcal{A}=\{|\alpha_1|,|\alpha_1|+|\alpha_2|,\ldots,|\alpha_1|+\cdots+|\alpha_{t-1}|\}$ and $\Delta=\sum\limits_{j=1}^{m} \Delta(\Gamma_j)$. Then we have 
	\item{\rm(i)} For $1\leq i \leq m-1$, 
\begin{align}
		b_{i+1}-b_{i}=\begin{cases}
			\Delta, & \text{if \,$i \notin \mathcal{A}$};\\
			\Delta-2m\cdot (-1)^{i}r_{i}, & \text{if \,$i \in \mathcal{A}$}.
		\end{cases} \label{difference}
	\end{align}
	\item{\rm(ii)} If $|\alpha_{d_2}|,\ldots,|\alpha_{d_s}|$ are all odd numbers among $|\alpha_{1}|,|\alpha_{2}|,\ldots,|\alpha_{t}|$ with $1 \leq d_1 < d_2 < \cdots < d_s \leq t$, then 
	\[\Delta=2\sum_{i=1}^{s} (-1)^{d_i+(i-1)} \cdot r_1.\]
\end{theorem}

\begin{proof}
	By {\rm(i)} of Proposition \ref{p1Zmatrix}, we have $b_i=\sum\limits_{j=1}^{m} \gamma^{(j)}_i$. By {\rm(ii)} Proposition \ref{p1Zmatrix} and $\gamma^{(i)}_{i}=\gamma^{(i+1)}_{i+1}=0$, we have
	\begin{align}
		b_{i+1}-b_{i}&=\sum_{j=1}^{m} \gamma^{(j)}_{i+1}-\sum_{j=1}^{m} \gamma^{(j)}_i \notag\\ 
		&=\sum_{j=1}^{i-1} (\gamma^{(j)}_{i+1}-\gamma^{(j)}_i)+\gamma^{(i)}_{i+1}-\gamma^{(i+1)}_{i}+\sum_{j=i+2}^{m} (\gamma^{(j)}_{i+1}-\gamma^{(j)}_i) \notag\\
		&=\sum_{j=1}^{i-1} \Delta(\Gamma_{j})+z_{(i+1)(m-1)}-z_{i1}+\sum_{j=i+2}^{m} \Delta(\Gamma_{j}) \notag\\
		&=\sum_{j=1}^{i-1} \Delta(\Gamma_{j})+(-1)^{i}(2-m)\cdot r_{i}-(-1)^{i+1}(m-2)\cdot r_{i+1}+\sum_{j=i+2}^{m} \Delta(\Gamma_{j}) \notag\\
		&=\sum_{j=1}^{i-1} \Delta(\Gamma_{j})-(m-2)\cdot ((-1)^{i}r_{i}+(-1)^{i+1}r_{i+1})+\sum_{j=i+2}^{m} \Delta(\Gamma_{j}).\label{bidiff}
	\end{align}
	
	We note that $\sum\limits_{j=i+2}^{m} \Delta(\Gamma_{j})=0$ if $i=m-1$ in \eqref{bidiff}. Now we show \eqref{difference} holds by the following two cases.
	
	\textbf{Case 1}: $i \notin \mathcal{A}$.
	
	In this case, $r_i \cdot r_{i+1}=1$, $(-1)^{i}r_i+(-1)^{i+1}r_{i+1}=0$, and $\Delta(\Gamma_{i})+\Delta(\Gamma_{i+1})=0$. By \eqref{bidiff}, we have $b_{i+1}-b_{i}=\Delta$.

	\textbf{Case 2}: $i \in \mathcal{A}$.
	
	In this case, $r_i \cdot r_{i+1}=-1$, $(-1)^{i}r_i+(-1)^{i+1}r_{i+1}=2\cdot(-1)^{i}r_i$, and $\Delta(\Gamma_{i})+\Delta(\Gamma_{i+1})=4\cdot(-1)^{i}r_{i}$. By \eqref{bidiff}, we have $b_{i+1}-b_{i}=\Delta-2m(-1)^{i}r_{i}$.
	
	This completes the proof of (i).
	
		By $\Delta(\Gamma_{\ell})=2\cdot (-1)^{\ell}\cdot r_{\ell}$, we have
	\begin{align*}
		\Delta&=\sum_{i=1}^{m} \Delta(\Gamma_i)=2\sum_{i=1}^{m} (-1)^{i}\cdot r_{i}\\
		&=2 \left( \sum_{i=1}^{|\alpha_1|} (-1)^{i}\cdot r_{i}+\sum_{i=|\alpha_1|+1}^{|\alpha_1|+|\alpha_2|} (-1)^{i}\cdot r_{i}+\cdots+\sum_{i=|\alpha_1|+\cdots+|\alpha_{t-1}|+1}^{|\alpha_1|+\cdots+|\alpha_{t-1}|+|\alpha_t|} (-1)^{i}\cdot r_{i} \right)\\
		&=2 \sum_{i=1}^{s}\left(\sum_{j=|\alpha_{1}|+\cdots+|\alpha_{d_i-1}|+1}^{|\alpha_{1}|+\cdots+|\alpha_{d_i-1}|+|\alpha_{d_i}|}(-1)^{j}\cdot r_{j} \right)\\
		&=2\sum_{i=1}^{s}(-1)^{|\alpha_{1}|+\cdots+|\alpha_{d_i-1}|+1}\cdot r_{|\alpha_{1}|+\cdots+|\alpha_{d_i-1}|+1}\\
		&=2\sum_{i=1}^{s}(-1)^{|\alpha_{1}|+\cdots+|\alpha_{d_i-1}|+1}\cdot (-1)^{d_i-1} \cdot r_1\\
		&=2\sum_{i=1}^{s}(-1)^{i-1+1}\cdot (-1)^{d_i-1} \cdot r_1\\
		&=2\sum_{i=1}^{s} (-1)^{d_i+(i-1)} \cdot r_1.
	\end{align*}
	
	This completes the proof of (ii).
\end{proof}

\subsection{Proofs of Theorems \ref{evenLnbasicCR} and \ref{allLnstrongCR}}\label{subsecLn-pfeven}
\begin{proof}[{\bf Proof of Theorems \ref{evenLnbasicCR}}] Since $L_4$ and $L_6$ are basic strong CR tournament by Proposition \ref{L246}, we only need to prove that $L_n$ is a basic strong CR tournament for even $n \geq 8$. Hence, we assume that $n\geq 8$ in the following.
	
	Let $V(L_n)=\{v_1,v_2,\ldots,v_{n-1},v_n\}$ and $X=\{v_1,v_2,\ldots,v_{n-1}\}$, where $L_n[X]$ is transitive with $v_1\rightarrow v_2\rightarrow \cdots \rightarrow v_{n-1}$ and $\psi_{L_n}(v_n,X)=((-1)^0,(-1)^1,\ldots,(-1)^{n-2})$. By Lemma \ref{Lnbasic}, $L_n$ is a basic tournament. By Lemma \ref{Lnstrong}, if $L_n$ is a CR tournament, then $L_n$ is a strong CR tournament. Therefore, we only need to prove that $L_n$ is a CR tournament.
	
	Let $u$ be a non-CR vertex for $L_n$ with the dominating relation $\sigma=(r_1,r_2,\ldots,r_n)$, where $r_i=\theta_{L_n(u,\sigma)}(u,v_i)$. We only need show $L_n(u,\sigma) \notin \mathcal{D}_{n-1}$ by Definition \ref{defCR} and $L_n \in \mathcal{D}_{n-1} \backslash \mathcal{D}_{n-3}$.
	
	By Lemma \ref{CRveforLneven}, $\psi_{L_n(u,\sigma)}(u,X)=(\alpha_1,\ldots,\alpha_t)$ satisfy $3 \leq t \leq n-2$. Let $X(i,\alpha_i)$ ($i=1,2,\ldots,t$) denote the vertex subset of $X$ corresponding to $\alpha_i$.
	
	If $\alpha_1<0$, then there exists a switch of $L_n(u,\sigma)$ with respect to $W=\{u\}$, denoted by $L'_n(u,\sigma)$, such that $\psi_{L'_n(u,\sigma)}(u,X)=(-\alpha_1,\ldots,-\alpha_t)$. By Corollary \ref{switchDk}, if $L'_n(u,\sigma) \notin \mathcal{D}_{n-1}$, then $L_n(u,\sigma) \notin \mathcal{D}_{n-1}$. Therefore, without loss of generality, we can assume that $\alpha_1>0$.
	
	Let $S$ be the skew-adjacency matrix of $L_n(u,\sigma)$ with respect to the vertex ordering $v_n,u,v_1,\ldots,v_{n-1}$, i.e., 
	\begin{align}
		S =  \left[ \begin{array}{rrrrrrr}
			0 & a & 1 & -1 & \cdots & -1 & 1 \\
			-a & 0 & r_1 & r_2 & \cdots & r_{n-2} & r_{n-1}\\
			-1 & -r_1 & 0 & 1 & \cdots & 1 & 1\\
			1 & -r_2 & -1 & 0 & \cdots & 1 & 1\\
			\vdots &\vdots & \vdots & \vdots & \ddots & \vdots & \vdots \\
			1 & -r_{n-2} & -1 & -1 & \cdots & 0 & 1\\
			-1 & -r_{n-1} & -1 & -1 & \cdots & -1 & 0
		\end{array}\right], \label{detS}
	\end{align}
	where $a=\theta_{L_n(u,\sigma)}(v_n,u)=-r_n$.
	
	Let $m=n-1$. Then the $\{1,-1\}$-sequence $r=(r_1,r_2,\ldots,r_{n-1})=(\alpha_1,\ldots,\alpha_t)$, $Z(m,r)$ be the $Z$-matrix as in Section \ref{subsecLn-table}, $Z(m,r) \cdot J_{m-1}=(b_1,b_2,\ldots,b_m)^{\mathsf{T}}$. For convenience, we use $L_n(u,\sigma,v_i)$ to denote $L_n(u,\sigma)[(V(L_n) \cup \{u\}) \backslash \{v_i\}]$ for $1 \leq i \leq n-1$.
	
	\textbf{Claim 1:} $\det(L_n(u,\sigma,v_i))=(a+b_i)^2$ holds for $1 \leq i \leq n-1$.
	
	{\bf Proof of Claim 1:} Let $W \subseteq V(L_n)$ be defined as
	\begin{align*}
		W=\begin{cases}
			\{v_n\}, & \text{if \,$i=1$};\\
			\{v_1,v_2,\ldots,v_{i-1},v_n\}, & \text{if $2 \leq i \leq n-1$ and $i$ is odd};\\
			\{v_1,v_2,\ldots,v_{i-1}\}, &\text{if $2 \leq i \leq n-1$ and $i$ is even}.
		\end{cases}
	\end{align*}
	Then $L_n(u,\sigma,v_i)$ is switching equivalent to $L^{*}_n(u,\sigma,v_i)$ with respect to $W$ such that $\theta_{L^{*}_n(u,\sigma,v_i)}(v_n,u)=-a$ and $v_2 \rightarrow \cdots \rightarrow v_{n-1}$ if $i=1$; $\theta_{L^{*}_n(u,\sigma,v_i)}(v_n,u)=(-1)^{i}a$ and $v_{i+1} \rightarrow \cdots \rightarrow v_{n-1} \rightarrow v_1 \rightarrow \cdots \rightarrow v_{i-1}$ if $2 \leq i \leq n-1$. Moreover, $\psi_{L^{*}_n(u,\sigma,v_i)}(v_n,X \backslash \{v_i\})=((-1)^0,(-1)^1,\ldots,(-1)^{n-3})$ always holds. 
	
	When $i=1$, let $S^{*}(1)$ be the skew-adjacency matrix of $L^{*}_n(u,\sigma,v_1)$ with respect to $v_n,u,v_2,\ldots,v_{n-1}$. Then we have
	\begin{align}
		S^{*}(1) =  \left[ \begin{array}{rrrrrrr}
			0 & -a & 1 & -1 & \cdots & 1 & -1 \\
			a & 0 & r_2 & r_3 & \cdots & r_{n-2} & r_{n-1}\\
			-1 & -r_2 & 0 & 1 & \cdots & 1 & 1\\
			1 & -r_3 & -1 & 0 & \cdots & 1 & 1\\
			\vdots &\vdots & \vdots & \vdots & \ddots & \vdots & \vdots \\
			-1 & -r_{n-2} & -1 & -1 & \cdots & 0 & 1\\
			1 & -r_{n-1} & -1 & -1 & \cdots & -1 & 0
		\end{array}\right]_{n \times n}.
	\end{align} 
	
	When $2 \leq i \leq n-1$, let $S^{*}(i)$ be the skew-adjacency matrix of $L^{*}_n(u,\sigma,v_i)$ with respect to $v_n,u,v_{i+1},\ldots,v_{n-1}, v_1, \ldots, v_{i-1}$ ($v_n,u,v_1,\ldots,v_{n-1}$ if $i=n-1$). Then we have
	\begin{align}
		S^{*}(i) =  \left[ \begin{array}{rrrrrrrrr}
			0 & (-1)^{i}a & 1 & -1 & \cdots & (-1)^{n-i-2} & (-1)^{n-i-1} & \cdots & -1 \\
			-(-1)^{i}a & 0 & r_{i+1} & r_{i+2} & \cdots & r_{n-1} & -r_1 & \cdots & -r_{i-1}\\
			-1 & -r_{i+1} & 0 & 1 & \cdots & 1 & 1 & \cdots & 1\\
			1 & -r_{i+2} & -1 & 0 & \cdots & 1 & 1 & \cdots & 1\\
			\vdots &\vdots & \vdots & \vdots & \ddots & \vdots & \vdots & & \vdots\\
			-(-1)^{n-i-2} & -r_{n-1} & -1 & -1 & \cdots & 0 & 1 & \cdots & 1\\
			-(-1)^{n-i-1} & r_{1} & -1 & -1 & \cdots & -1 & 0 & \cdots & 1\\
			\vdots & \vdots & \vdots & \vdots &  & \vdots & \vdots & \ddots & \vdots\\
			1 & r_{i-1} & -1 & -1 & \cdots & -1 & -1 & \cdots & 0
		\end{array}\right].
	\end{align} 
	
	Clearly, $\det(S^{*}(i))=\det(L^{*}_n(u,\sigma,v_i))$=$\det(L_n(u,\sigma,v_i))$ for $1 \leq i \leq n-1$.
	
	By Proposition \ref{toolforS} and Definition \ref{defZmatrix}, for $1 \leq i \leq n-1$, we have
	\begin{align*}
		\det(S^{*}(i))&=\left((-1)^{i}a+\sum_{j=1}^{m-i}(-1)^j \cdot (m-2j) \cdot r_{i+j}+\sum_{j=m-i+1}^{m-1}(-1)^j \cdot (m-2j) \cdot (-r_{i+j-m})\right)^2\\
		&=\left(a+\sum_{j=1}^{m-i}(-1)^{i+j} \cdot (m-2j) \cdot r_{i+j}+\sum_{j=m-i+1}^{m-1}(-1)^{i+j} \cdot (m-2j) \cdot (-r_{i+j-m})\right)^2\\
		&=(a+\sum_{j=1}^{m-1}z_{ij})^{2}=(a+b_i)^2.
	\end{align*}
	
	This completes the proof of \textbf{Claim 1}.
	
	\textbf{Claim 2:} If $|b_i - b_j| \geq 2m+2$, then there exists $k \in \{i,j\}$ such that $\det(L_n(u,\sigma,v_k))> m^2=(n-1)^2$.
	
	{\bf Proof of Claim 2:} By Claim 1, $\det(L_n(u,\sigma,v_k))=(a+b_k)^2$ for $k \in \{i,j\}$. If $\det(L_n(u,\sigma,v_i)) \leq m^2$, then $|a+b_i| \leq m$ and we have
	\begin{align*}
		|a+b_j|&=|a+b_i+b_j-b_i| \geq |b_j-b_i|-|b_i+a| \geq |b_j-b_i|-m \geq m+2,
	\end{align*}
	it follows that $\det(L_n(u,\sigma,v_j))=(a+b_j)^2 > m^2=(n-1)^2$.
	
	This completes the proof of \textbf{Claim 2}.
	
	\textbf{Claim 3:} If $|\alpha_i|$ is even for some $2 \leq i \leq t-1$, then there exists $j$ such that $\det(L_n(u,\sigma,v_j)) > (n-1)^2$.
	
	{\bf Proof of Claim 3:}
	Let $|\alpha_{d_1}|, |\alpha_{d_2}|,\ldots,|\alpha_{d_s}|$ be all odd numbers among $|\alpha_{1}|,|\alpha_{2}|,\ldots,|\alpha_{t}|$, and $1 \leq d_1 < d_2 < \cdots < d_s \leq t$. Since $|\alpha_1| + |\alpha_2| +\cdots + |\alpha_t| =n-1$ is odd, $s$ is odd. By Theorem \ref{thm1Zmatrix}, $\Delta=\sum\limits_{k=1}^{m} \Delta(\Gamma_k)=2\sum\limits_{k=1}^{s} (-1)^{d_k+(k-1)} \cdot r_1 \neq 0$. Then $|\Delta| \geq 2$.
	
	If there exits $2 \leq i \leq t-1$ such that $|\alpha_i|$ is even, let $\Delta_1=b_{|\alpha_1|+\cdots+|\alpha_{i-1}|+1}-b_{|\alpha_1|+\cdots+|\alpha_{i-1}|}$ and $\Delta_2=b_{|\alpha_1|+\cdots+|\alpha_{i}|+1}-b_{|\alpha_1|+\cdots+|\alpha_{i}|}$. By Theorem \ref{thm1Zmatrix}, we have 
	\begin{align}
		\Delta_1=\begin{cases}
			\Delta+2m, & \text{if $(-1)^{|\alpha_1|+\cdots+|\alpha_{i-1}|}r_{|\alpha_1|+\cdots+|\alpha_{i-1}|}=-1$};\\
			\Delta-2m, &\text{if $(-1)^{|\alpha_1|+\cdots+|\alpha_{i-1}|}r_{|\alpha_1|+\cdots+|\alpha_{i-1}|}=1$},
		\end{cases} \label{delta1}
	\end{align}
	and
	\begin{align}
		\Delta_2=\begin{cases}
			\Delta+2m, & \text{if $(-1)^{|\alpha_1|+\cdots+|\alpha_{i}|}r_{|\alpha_1|+\cdots+|\alpha_{i}|}=-1$};\\
			\Delta-2m, &\text{if $(-1)^{|\alpha_1|+\cdots+|\alpha_{i}|}r_{|\alpha_1|+\cdots+|\alpha_{i}|}=1$}.
		\end{cases} \label{delta2}
	\end{align}
	
	Since $|\alpha_{i}|$ is even, we have
	\begin{align*}
		(-1)^{|\alpha_1|+\cdots+|\alpha_{i-1}|}r_{|\alpha_1|+\cdots+|\alpha_{i-1}|} \cdot (-1)^{|\alpha_1|+\cdots+|\alpha_{i}|}r_{|\alpha_1|+\cdots+|\alpha_{i}|}=-(r_{|\alpha_1|+\cdots+|\alpha_{i-1}|})^{2}=-1.
	\end{align*}
	
	Therefore, by \eqref{delta1} and \eqref{delta2}, we have $\{\Delta_1,\Delta_2\}=\{\Delta+2m,\Delta-2m\}$. Since $|\Delta| \geq 2$, we have $|\Delta_1| \geq 2m+2$ or $|\Delta_2| \geq 2m+2$. Then by Claim 2, there exists $j \in \{|\alpha_1|+\cdots+|\alpha_{i-1}|, |\alpha_1|+\cdots+|\alpha_{i-1}|+1\}$ or $j \in \{|\alpha_1|+\cdots+|\alpha_{i}|, |\alpha_1|+\cdots+|\alpha_{i}|+1\}$ such that $\det(L_n(u,\sigma,v_j))> m^2=(n-1)^2$.
	
	This completes the proof of \textbf{Claim 3}.
	
	Now we complete the remaining proof by showing $L_n(u,\sigma) \notin \mathcal{D}_{n-1}$ from the following two cases.
	
	\textbf{Case 1:} $t$ is odd.
	
	\textbf{Subcase 1.1:} There exists $i$ ($1 \leq i \leq t$) such that $|\alpha_i|$ is even.
	
	If there exits $2 \leq i \leq t-1$ such that $|\alpha_i|$ is even, then $L_n(u,\sigma) \notin \mathcal{D}_{n-1}$ by Claim 3. So we only consider the case that $|\alpha_i|$ is even for $i \in \{1,t\}$ and $|\alpha_k|$ is odd for all $2 \leq k \leq t-1$.
	
	Since $t$ is odd and $|\alpha_1| + \cdots + |\alpha_t|=n-1$ is odd, $|\alpha_1|$ and $|\alpha_t|$ must are even. Moreover, $\alpha_1 \alpha_t>0$. Note that $\alpha_1 > 0$ (by assumption), then $\alpha_t >0$. Let $W=X(t,\alpha_t) \cup \{u\}$. Then $L_n(u,\sigma)$ is switching equivalent to $L^{*}_n(u,\sigma)$ with respect to $W$ such that $X(t,\alpha_t) \rightarrow X(1,\alpha_1) \rightarrow \cdots \rightarrow X(t-1,\alpha_{t-1})$ in $L^{*}_{n}(u,\sigma)$, $\psi_{L^{*}_{n}(u,\sigma)}(v_n,X)=((-1)^{0},(-1)^{1},\ldots,(-1)^{n-2})$ and $\psi_{L^{*}_{n}(u,\sigma)}(u,X)=(\beta_1,\beta_2,\ldots,\beta_t)=(\alpha_{t},-\alpha_{1},\ldots,-\alpha_{t-1})$. Then $\beta_1 = \alpha_t >0$ and $|\beta_2|=|\alpha_1|$ is even.
	
	Let 
	\[
	v^{*}_{k}=\begin{cases}
		v_{|\alpha_1|+\cdots+|\alpha_{t-1}|+k}, & \text{if $1 \leq k \leq |\alpha_{t}|$};\\
		v_{k-|\alpha_t|}, & \text{if $k > |\alpha_{t}|$}.
	\end{cases}
	\]
	Then $v^{*}_{1} \rightarrow v^{*}_{2} \rightarrow \cdots \rightarrow v^{*}_{n-1}$ in $L^{*}_{n}(u,\sigma)$.
	
	By the proof of Claim $3$ and $|\beta_2|$ is even, there exists $j' \in \{|\beta_1|,|\beta_1|+1\}$ or $j' \in \{|\beta_1|+|\beta_2|,|\beta_1|+|\beta_2|+1\}$ such that $\det(L^{*}_n(u,\sigma,v^{*}_{j'})) > (n-1)^2$. Let $v_{j}=v^{*}_{j'}$. Since $L_{n}(u,\sigma)$ and $L^{*}_{n}(u,\sigma)$ are switching equivalent, we have $\det(L_n(u,\sigma,v_{j}))=\det(L^{*}_n(u,\sigma,v^{*}_{j'}))>(n-1)^2$, and then $L_n(u,\sigma) \notin \mathcal{D}_{n-1}$.
	
	\textbf{Subcase 1.2:} All $|\alpha_i|$ are odd, and $|\alpha_1|=|\alpha_2|=\cdots=|\alpha_t|=\frac{n-1}{t}=\frac{m}{t}$.
	
	Clearly, $m$ is not a prime number by $t \mid m$, and $\frac{m}{t} \geq 3$ by the facts that $3 \leq t \leq m-1$ and $\frac{m}{t}=|\alpha_1|$ is odd. By the assumption that $n=m+1 \geq 8$, we have $m=9$ or $m \geq 15$. 
	
	\textbf{Subcase 1.2.1:} $m = 9$.
	
	By direct computation and \eqref{detS}, we have
	\begin{align*}
		\det(L_n(u,\sigma)[V_1])=121>m^2=81, \text{where } V_1=\begin{cases}
			\{v_1,v_2,v_3,v_5,v_8,v_9,v_{10},u\}, & \text{if $a=1$};\\
			\{v_1,v_2,v_5,v_7,v_8,v_9,v_{10},u\}, & \text{if $a=-1$}.\\
		\end{cases}
	\end{align*}
	
	Therefore, $L_n(u,\sigma) \notin \mathcal{D}_{n-1}$.
	
	\textbf{Subcase 1.2.2:} $m \geq 15$.
	
	Let $q=\frac{m}{t}$ ($\geq 3$), $X_1=\{v_1,v_2,v_3,\ldots,v_{|\alpha_1|+1}\} \backslash \{v_2\}$, $Y=X \backslash X_1$, $Z=Y \cup \{v_n,u\}$ and $W=\{v_n\}$. For simplicity, we denote $L_n(u,\sigma)[Z]$ by $L_n(u,\sigma,Z)$.
	
	Let $L^{*}_n(u,\sigma,Z)$ be a switch of $L_n(u,\sigma,Z)$ with respect to $W$. Then $L^{*}_n(u,\sigma,Z)[Y]$ is transitive with $v_2 \rightarrow v_{q+2}  \rightarrow v_{q+3} \rightarrow \cdots \rightarrow v_{tq}=v_{n-1}$, $\psi_{L^{*}_n(u,\sigma,Z)}(v_n,Y)=((-1)^{0},\ldots,$ $(-1)^{q(t-1)-1})$, $\psi_{L^{*}_n(u,\sigma,Z)}(u,Y)=(1, (-1)^{1}(q-1),(-1)^2q,\ldots,(-1)^{t-1}q)=(y_1,y_2,\ldots,y_{q(t-1)})$, where $(y_1,y_2,\ldots,y_{q(t-1)})$ is a $\{1,-1\}$-sequence, and $\theta_{L^{*}_n(u,\sigma,Z)}(v_n,u)=a^{*}=-\theta_{L_n(u,\sigma)}(v_n,u)=-a$.
	
	Let $S^{*}_{Z}$ be the skew-adjacency matrix of $L^{*}_n(u,\sigma,Z)$ with respect to the vertex ordering $v_n,u,v_2,v_{q+2},\ldots,v_{n-1}$. Note that $y_{i} = -y_{q(t-1)-(i-1)}$ for $2 \leq i \leq \frac{q(t-1)}{2}$. By Proposition \ref{toolforS}, we have
	\begin{align*}
		\det(S^{*}_{Z})&=\left(a^{*}+\sum_{i=1}^{q(t-1)}(-1)^i \cdot (q(t-1)+1-2i)\cdot y_i\right)^2\\
		&=\left(a^{*}-2(qt-q-1)+\sum_{i=2}^{q(t-1)-1}(-1)^i \cdot (q(t-1)+1-2i)\cdot y_i \right)^2\\
		&=\left( a^{*}- 2(qt-q-1) \right)^2=\left( 2m -\frac{2m}{t} -2 +a\right)^2.
	\end{align*}
	
	Since $m \geq 15$ and $t \geq 3$, we have
	\begin{align*}
		2m -\frac{2m}{t} -2 + a \geq 2m(1-\frac{1}{t})-3 \geq \frac{4}{3}m-3 \geq m+2.
	\end{align*}
	Then $\det(S^{*}_{Z}) =\left( 2m -\frac{2m}{t} -2 +a\right)^2 \geq (m+2)^2 > m^2=(n-1)^2$, which implies that $L^{*}_n(u,\sigma) \notin \mathcal{D}_{n-1}$ and thus $L_n(u,\sigma) \notin \mathcal{D}_{n-1}$ by Corollary \ref{switchDk}.
	
	\textbf{Subcase 1.3:} All $|\alpha_i|$ are odd, and $\max\{|\alpha_1|,\ldots,|\alpha_t|\} > \min\{|\alpha_1|,\ldots,|\alpha_t|\}$.
	
	Since all $|\alpha_i|$ are odd, we have $s=t$ and $d_i=i$ for $1 \leq i \leq t$ in (ii) of Theorem \ref{thm1Zmatrix}, and $\Delta=2\sum\limits_{i=1}^{s} (-1)^{d_i+(i-1)} \cdot r_1=-2t$ by (ii) of Theorem \ref{thm1Zmatrix}. Let $|\alpha_{j}|=\max\{|\alpha_1|,\ldots,|\alpha_t|\}$ (if there is more than one such index $j$, choose any one of them). Then $|\alpha_{j}|=[\frac{m}{t}]+d$ with $d \geq 1$. Since $|\alpha_j|$ is odd, we have $|\alpha_{j}| \geq 3$.
	
	\textbf{Subcase 1.3.1:} $d \geq 2$.
	
	Since $d \geq 2$, we have $|\alpha_{j}|-1=[\frac{m}{t}]+d-1 \geq [\frac{m}{t}]+1 > \frac{m}{t}$. Then by Theorem \ref{thm1Zmatrix} and $\Delta=-2t$, we have
	\begin{align*}
		&|b_{|\alpha_1|+\cdots+|\alpha_{j}|}-b_{|\alpha_1|+\cdots+|\alpha_{j-1}|+1}|\\
		=&|(b_{|\alpha_1|+\cdots+|\alpha_{j}|}-b_{|\alpha_1|+\cdots+|\alpha_{j}|-1})+ \cdots +(b_{|\alpha_1|+\cdots+|\alpha_{j-1}|+2}-b_{|\alpha_1|+\cdots+|\alpha_{j-1}|+1})|\\
		=&|\Delta+\cdots+\Delta|=|(|\alpha_{j}|-1)\Delta|>|\frac{m}{t} \Delta|=2m.
	\end{align*}
	
	Since $|(|\alpha_{j}|-1)\Delta|$ is even, we have $|b_{|\alpha_1|+\cdots+|\alpha_{j}|}-b_{|\alpha_1|+\cdots+|\alpha_{j-1}|+1}| \geq 2m+2$. Then by Claim 2, there exists $k \in \{|\alpha_1|+\cdots+|\alpha_{j}|,|\alpha_1|+\cdots+|\alpha_{j-1}|+1\}$ such that $\det(L_n(u,\sigma,v_k))> m^2=(n-1)^2$. Therefore, $L_n(u,\sigma) \notin \mathcal{D}_{n-1}$.
	
	\textbf{Subcase 1.3.2:} $d=1$.
	
	Since $|\alpha_j|$ is odd and $d=1$, $[\frac{m}{t}]$ is even. Let $m=2qt+p$, where $1 \leq p < t$. Choose $i^{*}$ so that $|\alpha_{i^{*}}|=\min\{|\alpha_1|,\ldots,|\alpha_t|\}=h$ (if there is more than one such index $i^{*}$, choose any one of them). Clearly, $h \leq 2q-1$.
	
	\textbf{Subcase 1.3.2.1:} $i^{*}=2$.
	
	By Theorem \ref{thm1Zmatrix}, we have
	\begin{align}
		&|b_{|\alpha_1|+|\alpha_{2}|+1}-b_{|\alpha_1|}| \notag \\
		=&|(b_{|\alpha_1|+|\alpha_{2}|+1}-b_{|\alpha_1|+|\alpha_{2}|})+ (b_{|\alpha_1|+|\alpha_{2}|}-b_{|\alpha_1|+|\alpha_{2}|-1})+ \cdots  +(b_{|\alpha_1|+1}-b_{|\alpha_1|})| \notag \\
		=&|2m+\Delta + (h-1)\Delta +2m + \Delta| = |4m+(h+1)\Delta| \notag \\
		=&|8qt+4p-2(h+1)t| =|(8q-2h-2)t + 4p|. \label{s1one}
	\end{align}
	
	Since $h \leq 2q-1$ and $p \geq 1$, we have
	\begin{align}
		&(8q-2h-2)t + 4p \geq (8q-4q+2-2)t + 4p \geq 4qt+2p+2  = 2m+2. \label{s1two}
	\end{align}
	
	By \eqref{s1one} and \eqref{s1two}, we have
	\begin{align}
		|b_{|\alpha_1|+|\alpha_{2}|+1}-b_{|\alpha_1|}| \geq  2m+2. \label{s1three}
	\end{align}
	
	Then by Claim 2 and \eqref{s1three}, there exists $k \in \{|\alpha_1|+|\alpha_{2}|+1, |\alpha_1|\}$ such that $\det(L_n(u,\sigma,v_k))> m^2=(n-1)^2$. Therefore, $L_n(u,\sigma) \notin \mathcal{D}_{n-1}$.
	
	\textbf{Subcase 1.3.2.2:} $i^{*} \neq 2$.
	
	Let 
	\[W=\begin{cases}
		X(t,\alpha_t) \cup \{u,v_n\}, & \text{if $i^{*}=1$};\\
		X(1,\alpha_1) \cup \cdots \cup X(i^{*}-2,\alpha_{i^{*}-2}) \cup \{u,v_n\}, & \text{if $i^{*} \geq 3$ and $i^{*}$ is odd}; \\
		X(1,\alpha_1) \cup \cdots \cup X(i^{*}-2,\alpha_{i^{*}-2}), & \text{if $i^{*} \geq 3$ and $i^{*}$ is even}.
	\end{cases}\]
	
	Then $L_n(u,\sigma)$ is switching equivalent to $L^{*}_n(u,\sigma)$ with respect to $W$ such that $L^{*}_n(u,\sigma)[X]$ is transitive with $X(t,\alpha_t) \rightarrow X(1,\alpha_1) \rightarrow \cdots \rightarrow X(t-1,\alpha_{t-1})$ if $i^{*}=1$, and $X(i^{*}-1,\alpha_{i^{*}-1}) \rightarrow X(i^{*},\alpha_{i^{*}}) \rightarrow \cdots \rightarrow X(t,\alpha_{t}) \rightarrow X(1,\alpha_{1}) \rightarrow \cdots \rightarrow X(i^{*}-2,\alpha_{i^{*}-2})$ if $i^{*} \geq 3$. Moreover, $\psi_{L^{*}_n(u,\sigma)}(v_n,X)=((-1)^{0},\ldots,(-1)^{m-1})$, and \[\psi_{L^{*}_n(u,\sigma)}(u,X)=(\beta_1,\beta_2,\ldots,\beta_{t})=\begin{cases}
		(\alpha_{t},-\alpha_1,\ldots,-\alpha_{t-1}), & \text{if $i^{*}=1$}; \\
		(-\alpha_{i^{*}-1},-\alpha_{i^{*}},\ldots, -\alpha_{t}, \alpha_{1}, \ldots, \alpha_{i^{*}-2}), & \text{if $i^{*} \geq 3$ and $i^{*}$ is odd}; \\
		(\alpha_{i^{*}-1},\alpha_{i^{*}},\ldots,\alpha_{t}, -\alpha_{1}, \ldots, -\alpha_{i^{*}-2}), & \text{if $i^{*} \geq 3$ and $i^{*}$ is even}.
	\end{cases}\]
	where $|\beta_2|=|\alpha_{i^{*}}|=h \leq 2q-1$, and $\beta_1 > 0$. Then $L^{*}_n(u,\sigma) \notin \mathcal{D}_{n-1}$ by Subcase 1.3.2.1, it follows that $L_n(u,\sigma) \notin \mathcal{D}_{n-1}$ by Corollary \ref{switchDk}.
	
	Therefore, $L_n(u,\sigma) \notin \mathcal{D}_{n-1}$ when $d=1$.
		
		Combining the above subcases, $L_n(u,\sigma) \notin \mathcal{D}_{n-1}$ when $t$ is odd.
		
		\textbf{Case 2:} $t$ is even.
		
		Let
		\[W=\begin{cases}
			X(t,\alpha_t)\cup \{v_n\}, & \text{if $\alpha_t$ is odd};\\
			X(t,\alpha_t), & \text{if $\alpha_t$ is even}.
		\end{cases}\]
		Then there exists a switch $L^{*}_n(u,\sigma)$ of $L_n(u,\sigma)$ with respect to $W$ such that $L^{*}_n(u,\sigma)[X]$ is transitive with $X(t,\alpha_t) \rightarrow X(1,\alpha_1) \rightarrow \cdots \rightarrow X(t-1,\alpha_{t-1})$, $\psi_{L^{*}_n(u,\sigma)}(v_n,X)=((-1)^{0},(-1)^{1},$ $\ldots,(-1)^{m-1})$, $\psi_{L^{*}_n(u,\sigma)}(u,X)=(|\alpha_t|+\alpha_1,\alpha_2,\ldots,\alpha_{t-1})$, where  $\alpha_1 > 0$ by assumption.
		
		Let $\beta_1=|\alpha_t|+\alpha_1$, $\beta_i=\alpha_i$ for $2 \leq i \leq t-1$. Then $\psi_{L^{*}_n(u,\sigma)}(u,X)=(\beta_1,\ldots,\beta_{t-1})$ with odd $t-1$ and $\beta_1 > 0$. By Case 1, we have $L^{*}_n(u,\sigma) \notin \mathcal{D}_{n-1}$, then $L_n(u,\sigma) \notin \mathcal{D}_{n-1}$ by Corollary \ref{switchDk}.
		
		Therefore, $L_n(u,\sigma) \notin \mathcal{D}_{n-1}$ when $t$ is even.
		
		Combining the above arguments, we complete the proof.
	\end{proof}
	
	\begin{proof}[{\bf Proof of Theorem \ref{allLnstrongCR}}]
	By Proposition \ref{L246}, $L_2$ is a strong CR tournament. Note that $L_3 \in \mathcal{D}_1$. Then a $1$-transitive blowup of $L_3$ is switching equivalent to a transitive $4$-tournament by Theorem \ref{D1} and Corollary \ref{crllblowupclass}. It is easy to check that a transitive $4$-tournament is a CR tournament (by direct checking). Thus a $1$-transitive blowup of $L_3$ is a CR tournament by Theorem \ref{thforswiso}, it follows that $L_3$ is a strong CR tournament by Proposition \ref{strongCR1}. Now we consider $n \geq 4$.
	
	If $n$ is even, then by Theorem \ref{evenLnbasicCR}, $L_n$ is a strong CR tournament.
	
	If $n$ is odd, then $L_n$ is switching equivalent to a $1$-transitive blowup of $L_{n-1}$ by Lemma \ref{Lnadd1switch}, and thus a $1$-transitive blowup of $L_n$ is switching equivalent to a transitive blowup of $L_{n-1}$ by Lemma \ref{lemtranblowswitcheq}. Then by Theorems \ref{thforswiso}, \ref{maintheorembasic} and \ref{evenLnbasicCR}, a $1$-transitive blowup of $L_n$ is a CR tournament, it follows that $L_n$ is a strong CR tournament by Proposition \ref{strongCR1}.
	
	This completes the proof.	
	\end{proof}

\section{An answer to Question \ref{queanyDktwo} and further questions}\label{sec-solutionquestion}
\hspace{1.5em}In this section, by using Theorems \ref{maintheorembasic} and \ref{evenLnbasicCR}, we show that a necessary and sufficient condition for Question \ref{queanyDktwo} is $T\in \xi(L_{k+1})$, and we propose several questions for further research.

\begin{theorem}\label{solutionque}
	Let $T \in \mathcal{D}_{k} \backslash \mathcal{D}_{k-2}$, where odd $k \geq 7$. Then $T$ is switching equivalent to a transitive blowup of $L_{k+1}$ if and only if $T \in \xi(L_{k+1})$.
\end{theorem}
\begin{proof}
	If $T$ is switching equivalent to a transitive blowup of $L_{k+1}$, then it is clear that $T \in \xi(L_{k+1})$. If $T \in \xi(L_{k+1})$, then by $T \in \mathcal{D}_{k} \backslash \mathcal{D}_{k-2}$, Theorems \ref{evenLnbasicCR} and \ref{maintheorembasic}, $T$ is switching equivalent to a transitive blowup of $L_{k+1}$. This completes the proof.
\end{proof}

In Section \ref{sec-allLn}, we show all $L_n$ are strong CR tournaments. Note that {\rm (i)} of Theorem \ref{maintheorembasic} requires $H$ to be a strong CR tournament, rather than merely a CR tournament. A natural question is that which CR tournaments are strong CR tournaments.

However, as shown in Section \ref{sec-allLn}, all $L_n$ are strong CR tournaments. Moreover, after examining several low-order CR tournaments, we have not found any instance where a CR tournament fails to be a strong CR tournament. Hence we further propose the following question.

\begin{question}\label{allCRstrong}
	Is every CR tournament a strong CR tournament?
\end{question}

\begin{question}\label{sufficientCR}
	If there exists a CR tournament that is not a strong CR tournament, find some sufficient conditions, necessary conditions, necessary and sufficient conditions for a CR tournament to be a strong CR tournament.
\end{question}

\begin{remark}\label{remarkstrongCR}
	If all CR tournaments are strong CR tournaments, then it is easy to see that the property of ``being a CR tournament'' is an invariant under transitive blowup operation.
\end{remark}

By Theorem \ref{Dthree} and Theorem \ref{D5character}, $\mathcal{D}_3 \backslash \mathcal{D}_1$ and $\mathcal{D}_5 \backslash \mathcal{D}_3$ are characterized by the transitive blowups of a basic tournament ($L_4$ and $L_6$, respectively). Furthermore, we propose the following question.

\begin{question}\label{Dkfinitebaisc}
	Let $k \geq 7$ be a positive odd integer. Can we find a finite number of basic tournaments $H_1,\ldots,H_m \in \mathcal{D}_{k} \backslash \mathcal{D}_{k-2}$ such that a tournament $T \in \mathcal{D}_{k} \backslash \mathcal{D}_{k-2}$ if and only if $T$ is switching equivalent to a transitive blowup of some $H_i \in \{H_1,\ldots,H_m\}$?
\end{question}

\begin{remark}
	By Theorem \ref{maintheorembasic}, a natural question arises: can we find a finite set of basic strong CR tournaments $H_1,\ldots,H_m \in \mathcal{D}_{k} \backslash \mathcal{D}_{k-2}$ such that a tournament $T \in \mathcal{D}_{k} \backslash \mathcal{D}_{k-2}$ if and only if $T$ is switching equivalent to a transitive blowup of some $H_i$? However, this turns out to be false. For example, the tournament $T'$ shown in \eqref{exampleT} is a basic tournament and satisfies $T'\in \mathcal{D}_7 \backslash \mathcal{D}_5$, but $T'$ is not a strong CR tournament, and it can not be switching equivalent to a transitive blowup of any tournament with order fewer than $6$. Therefore, $\mathcal{D}_{7}\backslash \mathcal{D}_{5}$ cannot be characterized by a finite set of transitive blowups of basic strong CR tournaments. Interestingly, if we repeatedly add non-CR vertices to $T'$ and its successors while ensuring the generated tournament remains in $\mathcal{D}_{7}\backslash \mathcal{D}_{5}$, then after finitely many steps, the resulting tournament becomes a basic strong CR tournament (verified by computer checking). This suggests that Question \ref{Dkfinitebaisc} is both interesting and worthy of further exploration.
\end{remark}

\vspace{2em}
\noindent
{\bf Acknowledgments}\,

This work is supported by the National Natural Science Foundation of China (Grant Nos.12371347, 12271337).

\end{document}